\documentclass[10pt]{amsart}

\usepackage[utf8]{inputenc}
\usepackage[T1]{fontenc}
\usepackage[pdftex]{graphicx}
\usepackage{latexsym}
\usepackage{amsmath,amssymb,amsthm}
\usepackage{mathtools}
\usepackage{enumitem}
\usepackage{faktor}
\usepackage{enumerate}
\usepackage{hyperref}
\usepackage{tikz}
\usepackage{xcolor}

\newtheorem{theorem}{Theorem}[section]
\newtheorem{corollary}[theorem]{Corollary}
\newtheorem{proposition}[theorem]{Proposition}
\newtheorem{lemma}[theorem]{Lemma}
\theoremstyle{definition}

\newtheorem*{remark}{Remark}
\newtheorem{example}[theorem]{Example}

\newcommand{\F}{{F}}
\newcommand{\red}{\color{red}}
\newcommand{\Q}{\mathcal{Q}}

\newcommand{\Hmm}[1]{\leavevmode{\marginpar{\tiny%
			$\hbox to 0mm{\hspace*{-0.5mm}$\leftarrow$\hss}%
			\vcenter{\vrule depth 0.1mm height 0.1mm width \the\marginparwidth}%
			\hbox to 0mm{\hss$\rightarrow$\hspace*{-0.5mm}}$\\\relax\raggedright #1}}}
	
	\newcommand{\eat}[1]{}

\begin{document}
	\author[P. Bartmann, M. Keller]{Philipp Bartmann, Matthias Keller}
	
	\address{Philipp Bartmann: Institut f\"ur Mathematik, Universit\"at Potsdam
		14476  Potsdam, Germany}

        \email{philipp.bartmann@uni-potsdam.de}
    \address{Matthias Keller: Israel Institute of Advanced Studies (IIAS) The Hebrew University of Jerusalem, Feldman bldg. Edmond J. Safra Campus, Givat Ram, Jerusalem, Israel;
    Institut f\"ur Mathematik, Universit\"at Potsdam
		14476  Potsdam, Germany}
	\email{matthias.keller@uni-potsdam.de}

	\title[]{The Heat equation and independence of the spectrum of the Hodge Laplacian on $\ell^p$}

\begin{abstract}
We study the heat equation associated to the Hodge Laplacian on simplicial complexes. Using recently developed techniques for magnetic Schrödinger operators, we prove Davies-Gaffney-Grigoryan type estimates for the kernel of the heat semigroup on $\ell^2,$ which we then use to extend the  semigroup to $\ell^p$ for $p\in[1,\infty]$ under suitable curvature and volume growth conditions. Furthermore, we establish $p$-independence of the Hodge Laplacian spectrum under the assumption of form bounded curvature and uniform subexponential volume growth. While the main focus of the paper is the Hodge Laplacian on simplicial complexes, the results are indeed proven for general positive magnetic Schrödinger operators on graphs.
  
\end{abstract}
\maketitle
\tableofcontents

\section{Introduction}
In Riemannian geometry, the Hodge Laplacian is most famously known for its role in the celebrated Hodge theorem, which identifies its kernel with the de Rham cohomology of compact manifolds. In addition, its spectrum encodes geometric information, and so in the past century it has been extensively studied and defined the field of spectral geometry. In recent years there has been a growing interest in Hodge theory for various discrete spaces \cite{C,F,GrigPath,HJ,jost2026spectra,MHJ} going back to the seminal work of Eckmann \cite{Eckmann1944}. One natural way to understand how geometry interacts with the analytic properties of the Hodge Laplacian is through the associated heat equation, by examining how heat evolves and diffuses over time on a given space. The $L^2$-theory of the heat equation is classical  and well-understood via the spectral theorem both in the continuous \cite{Chavel,Dodziuk,Do83, DM97,deRham,Strichartz} and the discrete setting for graphs \cite{Barlow,BHY1,KLW,Grigoryanbook,Wojc}, see also recent work on simplicial complexes and hypergraphs \cite{HL,MugnoloHyper,MKBJ}. Especially, the works of Dodziuk bridged the gap between the continuum and the discrete setting, see e.g. \cite{Dodziuk,DM98}. For manifolds there is  a substantial body of literature on more general $L^p$-spaces \cite{CC,Ch,ChL,Magniez,Strichartz} as well, which reveal even more interesting and subtle phenomena such as $L^p$-cohomology, $L^p$-Hodge decompositions and boundedness of Riesz-transforms in $L^p$ as well as their connection to the heat semigroup. In particular, the behavior of spectra of semigroup generators in $L^p$ and their dependence on $p$ has attracted considerable attention for Laplace and Schrödinger-type operators both on functions \cite{CKK,HempelVoigt1986,HempelVoigt1987, Sturm} and forms as well \cite{Ch}. In \cite{Braun} the theory is generalized to a broader class of metric spaces.\\
For discrete spaces such as simplicial complexes, the $\ell^2$-theory has been adapted in many cases and many results from the continuous setting admit discrete analogues. However, this was mostly done for finite simplicial complexes \cite{Dodziuk,HJ} and only recently people became interested in the infinite setting as well,\cite{ATH,C,ennaceur_jadlaoui_hodge,ennaceur_jadlaoui_geometric,HL,Masamune}. In particular the $\ell^p$-theory seems to be  largely unexplored beyond the graph setting \cite{BauerHuaKeller,CHQSZ,CCH}. An notable exception here is \cite{JM} which, however, considers  also  finite complexes only.\\
In this paper, we take a step towards filling this gap by exploring the $\ell^p$-theory of the heat semigroup as well as the spectral theory of the $\ell^p$-generators and aim to lay the groundwork for further developments. To do so, we first establish a Davies-Gaffney-Grigoryan type estimate for the kernel of the heat semigroup. This generalizes the result for bounded Laplacians in \cite{HL} and answers the question posed therein, concerning unbounded operators. We will then use these estimates to extend the semigroup from $\ell^2$ to $\ell^p$, $1\leq p\leq\infty,$ under appropriate volume growth assumptions with respect to an intrinsic metric. This allows us to solve the respective heat equation with initial conditions in $\ell^p.$ Besides that, we will explore two more criteria that allow to extend semigroups to $\ell^p,$ namely boundedness of Laplacians in $\ell^p$ and form boundedness of the Forman curvature.\\
 Building on these results, we establish the $p$-independence of the Hodge Laplacian spectrum given  uniform subexponential volume growth and curvature bounds. Our findings not only prove discrete analogues of  classical results from Riemannian geometry \cite{Braun,Ch, Sturm} but also hold under rather weak assumptions. In particular, we do not need to assume uniform lower curvature bounds but only form bounded curvature  nor do we need any apriori heat kernel bounds beyond the Davies-Gaffney-Grigory'an bound, which is proven in this paper and always holds for a given intrinsic metric. Furthermore, in the combinatorial setting, which we discuss below, we do not even need the curvature assumption at all.\\
To set the stage, we consider the Hodge Laplacian on a combinatorial simplicial complex $\Sigma$  with standard weights, which is a natural instance applicable to our more general results. For details on the general setting of weighted simplicial complexes we refer to Section~\ref{s:weighted} recalling the setting from \cite{BK}.
So, let $  \Sigma$ be a 
combinatorial simplicial complex with standard weights
and consider the Hodge Laplacian acting on a direct sum over all 
$k$-forms $\omega=\oplus_{k}\omega_k$ 
\[\Delta^H = \delta\partial + \partial\delta,\]
where $\delta$ and $\partial$ are the coboundary and boundary operator respectively.
In \cite{BK}, it was shown that the quadratic form \[Q^H(\omega)=\|\delta \omega\|^2_2 + \|\partial \omega\|^2_2    \] defined on compactly supported forms is closable and thus gives rise to a positive self-adjoint operator $\Delta^H_2$ on $\ell^2$.

We are interested in two geometric concepts:  curvature bounds and volume growth bounds. For curvature, we consider the Forman curvature $c^H,$ which arises naturally as a potential from the discrete Bochner-Weitzenböck formula for $\Delta^H$, \cite{F}. Specifically, 
for $\tau\in  \Sigma$ with $k = \dim(\tau)$, we have , cf. \cite[Section~6.1]{BK}
\begin{align*}
	c^H(\tau)  = 2(k+1)+(k+2)\#\lbrace\sigma\succ\tau\rbrace-\sum_{\rho\prec\tau}\#\lbrace \tau'\succ \rho\rbrace,
\end{align*}
where $\prec$ and $\succ$ refer to the incidence relation on the complex.
Even though $\Delta^H_2$ is a positive operator, $c^H$ may take negative values. We consider the situation where the negative part $c^H_-= -(c^H \wedge 0)$ is form bounded with respect to $Q^H$, i.e., where one has $M,C\geq 0$ such that for all compactly supported $\varphi\in C_c(\Sigma)$
\begin{equation*}\label{cformbound}
     \| \sqrt{c_-^H}\varphi\|_2^2\leq M {Q}^H(\varphi)+C\Vert\varphi\Vert^2_2.
 \end{equation*}
In this case, we say that $\Sigma$ has \emph{form bounded curvature} with \emph{bound $M$}.  This assumption is natural, as it generalizes uniform lower curvature bounds while maintaining control over the analytic properties of $\Delta^H_2$.

Regarding volume growth, we consider the growth of balls with respect to 
the combinatorial graph metric on the $1$-skeleton of $\Sigma$ given by 
\begin{equation*}	d(v,w)=\min\{ n\mid \mbox{there is a path } {v=v_{0}\sim\ldots\sim v_{n}=w}\}.\end{equation*}
Here $v\sim w$ means the usual adjacency relation on graphs, i.e., $v$ and $w$ are connected by an edge in the 1-skeleton of $  \Sigma.$ 
We assume that the $1$-skeleton is connected and consider the growth of balls $B_r(v) = \{w\in  \Sigma_0\mid d(v,w)\leq r\}$ with respect to $d$. We say that the complex has \emph{(uniform) exponential growth} with  \textit{growth rate} $\nu\ge 0$ if there is $C>0$ such that for all $v\in  \Sigma_0$, $r>0$,  
\[\#B_r(v)\leq C\mathrm{e}^{\nu r}.\]
Similarly, we say that it has \textit{(uniform) subexponential volume growth} if for all $\varepsilon>0$ there is $C_\varepsilon>0$ such that for all $v\in  \Sigma_0,r>0$
\[\#B_r(v) \leq C_\varepsilon\mathrm{e}^{\varepsilon r} .\]
One may  wonder why the combinatorial distance is sufficient for our purposes here. The reason is that  exponential volume bounds automatically imply boundedness in the combinatorial setting, see Proposition~\ref{boundedness} and Lemma~\ref{LemmaVolGrowthSC}. This is different for weighted simplicial complexes as discussed in Section~\ref{s:weighted}.

Our goal is to study the heat equation for $t>0$ and $x\in  \Sigma$
\begin{equation} \tag{HE} \label{HE}
     -\Delta^H \omega_t (x)= \frac{d}{dt}\omega_t(x)     
\end{equation}
for various initial conditions $\omega_0.$ For $\omega_0\in \ell^2$, this is classically covered by the spectral theorem via the semigroup $\omega_t = \mathrm{e}^{-t\Delta^H_2}\omega_0$. 
We say one can \emph{solve \eqref{HE} on $\ell^p$} for $p\in [1,\infty]$ if for all initial conditions $\omega_0\in\ell^p$ there is a  function
$[0,\infty)\rightarrow\ell^p$, $t\mapsto \omega_t$ such that $t\mapsto \omega_t(x)$ is  continuous on $[0,\infty)$, differentiable on $(0,\infty)$ and satisfies \eqref{HE} for all $t>0$ and $x\in \Sigma$.
Moreover, we say that \eqref{HE} is \emph{solved by a strongly continuous semigroup} on $\ell^p$ if there is a strongly continuous semigroup $S_p$ on $\ell^p$ such that a solution of \ref{HE} is given by $\omega_t = S_p(t)\omega_0$ for all initial conditions $\omega_0\in\ell^p.$
We prove the following special case of our main results. 

\begin{theorem}[Heat equation]\label{thm:HE}
    Let $  \Sigma$ be a combinatorial simplicial complex. 
    \begin{itemize}
        \item [(a)] Assume $  \Sigma$ has form bounded curvature with bound $M\geq 0$. Then,  \eqref{HE} is solved by a strongly continuous semigroup on $\ell^p$ for all 
    $p$ in the interval $ I=  (2\frac{M+1}{M}-2\frac{\sqrt{M+1}}{M},2\frac{M+1}{M}+2\frac{\sqrt{M+1}}{M})$, where $I=(1,\infty)$ if $M=0$.
 \item [(b)] Assume $  \Sigma$ has at most exponential growth. Then,  \eqref{HE} is solved by a strongly continuous semigroup on $\ell^p$ for all 
    $p\in  [1,\infty]$. 
    \end{itemize}
\end{theorem}

Indeed, the heat equation \eqref{HE} can also be solved for $p$ at the boundary  of the interval $I$ in (a)  if we additionally assume that the initial condition $\omega_0$ lies in the domain of the generator of the semigroup. More specifically, our results give that the $\ell^2$-semigroup extends to a strongly continuous semigroup on $\ell^p$ for $p\neq\infty$ in the corresponding closed intervals. However, for values $p$ at the boundary of $I$ the resulting semigroups are not necessarily analytic and therefore do not automatically map into the domain of their generators.
We denote the generator of this semigroup by $-\Delta^H_p$ and the spectrum by $\sigma(\Delta^H_p).$ 

Using these results about the heat equation, we establish $p$-independence of the spectrum of the Hodge Laplacian under suitable geometric conditions. Specifically, we prove the following theorem.

\begin{theorem}
[$p$-independence of the spectrum]\label{thm:p-independence}
    Let $  \Sigma$ be a combinatorial simplicial complex with 
    subexponential growth. Then, for all 
    $p\in [1,\infty]$, we have
    \[\sigma(\Delta^H_p) = \sigma(\Delta^H_2).\]
\end{theorem}

In \cite{BK}, we developed a framework for Hodge Laplacians on weighted simplicial complexes and established their representation as signed Schrödinger operators. This perspective not only unifies various discrete Laplacians but also opens the door to applying techniques from the theory of Schrödinger operators. This will be instrumental in our analysis of the heat equation and the spectral properties of the Hodge Laplacian. More precisely, we will prove the results above in the more general setting of positive magnetic Schrödinger operators whose negative part of the potential is form bounded. These will be introduced in the next section, where we prove crucial Davies-Gaffney-Grigoryan type estimates, which we use in Section~\ref{s:semigroupsonlp} to extend the generated semigroups on $\ell^2$ to $\ell^p$ spaces and solve the parabolic equation. Finally, in Section~\ref{s:p-independence}, we study $p$-independence of the spectrum. The results for the Hodge Laplacian stated above then follow as special cases in 
Section~\ref{s:weighted}.

We write $\#A$ or $|A|$ for the cardinality of a finite set $A$. Furthermore, $C$ or $C''$ denote constants which may change from line to line.

\section{Magnetic Schrödinger Operators and Semigroups on $\ell^2$}\label{s:magneticschrodinger} 
In this section, we consider magnetic Schrödinger operators on weighted graphs.  In \cite{BK}, we have shown that the Hodge Laplacian can be represented as a magnetic Schrödinger operator. Thus, the results of this section apply in particular to the Hodge Laplacian. 
For background on operators on graphs, we refer to \cite{KLW, GKS,S}.

\subsection{Magnetic Schrödinger Operators on Graphs}
Let $X$ be a discrete set and denote by $C(X)$ the space of all complex valued functions on $X.$ Respectively, denote by $C_c(X)$ the space of all finitely supported functions on $X.$ For a discrete set $A$ and an absolutely summable or positive function $f\in C(A)$, we fix the notation
\[\sum_A f =\sum_{x\in A} f(x).\]
A function $m:X\rightarrow(0,\infty)$ extends to a measure on $X$ of full support by setting  $m(A)=\sum_{  A}m $ for $A\subseteq X.$ We call $(X,m)$ a \textit{discrete measure space}. The $\ell^p$-spaces, $p\in [1,\infty)$ associated to $(X,m)$ are defined as usual by  
\begin{align*}
    \ell^p=\ell^p(X,m)=\Big\{ f\in C(X)\mid \Vert \omega\Vert_p^p=\sum_{X}m \vert f \vert^p <\infty\Big\}
\end{align*}
and
\begin{align*}
   \ell^\infty=\ell^\infty(X)=\Big\{f\in C(X)\mid \Vert \omega\Vert_\infty=\sup_{X}\vert f \vert <\infty\Big\}.
\end{align*}
The space $\ell^2$ is a Hilbert space with inner product  $\langle f,g\rangle = \sum_X mf\overline{g}.$\\
A graph $b$ over $(X,m)$, is a symmetric function   $b :X\times X\rightarrow[0,\infty)$  with zero diagonal satisfying \[\sum_{y\in X}b(x,y)<\infty\] for all $x\in X$. Moreover, we let $o: X\times X\rightarrow\mathbb{T} = \{z\in\mathbb{C}\mid \vert z\vert = 1\}$ be a function satisfying $o (x,y) = \overline{o(y,x)}$ for all $x,y\in X.$ We call $o$ a \textit{magnetic potential} on $X.$ Finally, let $c:X\rightarrow\mathbb{R}$ be an \emph{(electric) potential}. We call the triple $(b,o,c)$ over $(X,m)$ a \textit{magnetic Schrödinger graph}. We say that the graph is \emph{locally finite} if, for all $x\in X$, we have $\#\lbrace y\in X\mid b(x,y)> 0\rbrace<\infty.$\\
The \textit{magnetic Schrödinger operator} $\mathcal{H}$ associated to the triple $(b,o,c)$ over $(X,m)$ is defined on
\begin{align*}
    F=\{ f\in C(X)\mid \sum_{y\in X}b(x,y)\vert f(y)\vert <\infty~\text{for all}~x\in X\}
\end{align*}
and is acting as
\[\mathcal{H}f(x) = \frac{1}{m(x)}\sum_{y\in X}b(x,y)\left(f(x)-o(x,y)f(y)\right)+\frac{1}{m(x)}c(x)f(x).\]
To introduce a self-adjoint restriction on $\ell^2$, we define an associated sesquilinear form on $C_c(X)$. For a function $f\in C(X)$, we denote
\begin{align*}
    \nabla_{o}f (x,y) = f(x)-o(x,y)f(y)\quad\mbox{and}\quad \nabla f (x,y) = f(x)-f(y).
\end{align*}
Green's formula \cite[Lemma 2.1.]{GKS} then states that for $\varphi  \in C_c(X)$ and $f\in F$ one has
\begin{align*}
    \sum_X m\mathcal{H}f \overline{\varphi}= \frac{1}{2}\sum_{X\times X}b \nabla_o f\overline{\nabla_o \varphi }  + \sum_{ X}c  f \overline{\varphi }=\sum_X mf \overline{\mathcal{H}\varphi} .
\end{align*}
For $\varphi,\psi\in C_c(X)$, we let
\begin{align*}
    Q(\varphi,\psi) = \frac{1}{2}\sum_{X\times X}b \nabla_o \varphi\overline{\nabla_o \psi }  + \sum_{ X}c \varphi \overline{\psi }.
\end{align*}
We assume that $Q$ is \emph{positive}, i.e.,  for all $\varphi\in C_c(X)$
\begin{align*}
    Q(\varphi) := Q(\varphi,\varphi) \geq 0.
\end{align*}
Whenever $Q$ is closable on $\ell^2,$ we denote its closure again by $Q$ with domain $D(Q).$ The associated positive self-adjoint operator on $\ell^2$ is denoted by $H=H_2.$ \\

{\bf Assumption.} We assume $(b,o,c)$ is a magnetic Schrödinger graph over $(X,m)$ such that $Q$ is positive and closable on $\ell^2$.\\

For criteria for closability of $Q$, we refer to \cite{GKS}. One particular instance relevant to this paper are forms with form bounded potential. We say that the potential $c$ is \textit{form bounded} with respect to $Q$ \emph{with bounds} $M,C\geq 0$ if for all $\varphi\in C_c(X)$ 
\begin{equation*}\label{formbounded}
     \| \sqrt{c_-/m}\varphi\|_2^2\leq M Q(\varphi)+C\Vert\varphi\Vert^2_2,
 \end{equation*}
where $c_- = -(c\wedge 0)$ is the negative part of $c.$ In \cite[Proposition 2.8.]{GKS} it is shown that if $c$ is form bounded, then $Q$ is   closable in $\ell^2 $. Note that a form bounded potential corresponds to form bounded curvature for the Hodge Laplacian as indicated in the introduction.  However, positivity and closability of the form of the Hodge Laplacian are always satisfied, even without any assumption on the curvature, as discussed in \cite{BK}.

\subsection{The Magnetic Schrödinger Semigroup on $\ell^2$}
The spectral theorem gives rise to a strongly continuous contraction semigroup $S_2:[0,\infty)\rightarrow\mathcal{B}(\ell^2)$, $S_2 (t)f = \mathrm{e}^{-t H} f,$ where $\mathcal{B}(\ell^2)$ are the bounded operators on $\ell^2$, which is characterized for $t\geq 0$ by its integral kernel
 \[p_t: X \times X\rightarrow\mathbb{C},\qquad p_t(x,y) = \frac{\langle\mathrm{e}^{-tH} 1_x,1_y\rangle}{m(x)m(y)}.\]
Let $K\subseteq X$ be finite and $\ell^2(K,m) = C(K).$
We have the canonical inclusion $\iota_K:\ell^2(K,m)\rightarrow\ell^2,$ which is obtained by extending functions from $K$ to $ X$ by zero. Its adjoint $\iota^*_K$ is given by the restriction of functions in $\ell^2$ to $K.$ We define the operator $H_K$ acting on $\ell^2(K,m)$ as 
\[H_K f( x) = \frac{1}{m( x)}\sum_{ y\in K}b( x, y)\left( f( x)-o( x, y) f( y)\right)+\frac{1}{m( x)}c_K( x) f( x),\]
where $
c_K( x)= c( x)+\sum_{ y\in X\backslash K}b( x, y) .$
By construction, $H_K$ is the self-adjoint positive operator associated to the quadratic form $Q_K = Q\mid_{\ell^2(K,m)}$ and gives rise to the semigroup $\mathrm{e}^{-tH_K}.$ 

Now let $K_n$, $n\in\mathbb{N}$, be an exhaustion of $ X$ by finite subsets, i.e., finite subsets such that $K_n\subseteq K_{n+1}$ and $\bigcup_{n}K_n =  X.$ Then we have the following strong convergence of semigroups, shown in \cite[Proposition~2.20]{GKS}.
\begin{proposition}[Semigoup convergence]\label{semiconv} Assume that $Q$ is positive and closable on $\ell^2.$
    Let $t\geq 0$ and $K_n$ be an exhaustion. Then,
    \[\iota_{K_n}\mathrm{e}^{-t{H}_{K_n}}\iota_{K_n}^*\rightarrow\mathrm{e}^{-t H}~\text{strongly in}~\ell^2~\text{as}~n\rightarrow\infty.\]
\end{proposition}

\subsection{Davies-Gaffney-Grigoryan Lemma}

In this subsection, we prove a Davies-Gaffney-Grigoryan type lemma for magnetic Schrödinger operators. This extends the result of \cite{BHY1,BHY2}. 

Let us consider an \emph{intrinsic metric} $d$ with respect to  the graph $b$ over $( X,m),$ that is, a pseudo metric $d$ such that for all $x\in X$
\begin{align*}
    \frac{1}{m(x)}\sum_{y\in X}b(x,y)d(x,y)^2\leq 1.
\end{align*}
We define the \textit{jump size} $s$ of $d$ as \[s = \sup_{x\sim y}d(x,y),\] where $x\sim y$ means that the supremum runs over all pairs $x,y\in X$ with $b(x,y) >0.$ \\

{\bf Assumption.} We fix an intrinsic metric $d$ with finite jump size $s>0.$
\\

Let $\sigma(H_2)$ denote the $\ell^2$-spectrum of $H=H_2.$ Let \[\lambda_2=\inf\sigma(H_2).\]

\begin{theorem}[Davies-Gaffney-Grigoryan]\label{dgg}
    For all $f,g\in\ell^2$, $A=\mathrm{supp}\ f$, $B= \mathrm{supp}\ g$ one has
    \[\vert\langle\mathrm{e}^{-t H}f,g\rangle\vert\leq\mathrm{exp}\left(-\lambda_2 t-\frac{t}{s^2}\zeta \left(\frac{s d(A,B)}{t}\right)\right)\Vert f\Vert_2\Vert g\Vert_2,\]
    where \[\zeta(r) =  \left(r  \mathrm{arcsinh}\left({ r} \right)-\sqrt{1+ r^2}+1\right).\] 
\end{theorem}
By choosing characteristic functions $1_x,1_y,$ of vertices $x,y\in X$, Theorem~\ref{dgg} yields the estimate 
\[\vert p_t(x,y)\vert\leq \frac{\mathrm{exp}(-\lambda_2 t-(t/s^2)\zeta(sd(x,y)/t))}{\sqrt{m(x)m(y)}}.\]

The proof consists of two main steps. First, we prove a  decay estimate of an energy functional for finite subsets $K\subseteq X,$ which extends the result in \cite[Lemma~3.2]{BHY2} to magnetic Schrödinger operators. Secondly, we use the semigroup convergence from Proposition~\ref{semiconv} to extend the result to the full space $X$.


The next lemma was shown for graphs in \cite{BHY1,BHY2}, for simplicial complexes with bounded Laplacians in \cite{HL} and ultimately goes back to results on manifolds \cite{Davies,Gaffney}. We extend it to magnetic Schrödinger operators on finite subsets. We say that a function $\omega:X\rightarrow\mathbb{R}$ is Lipschitz with Lipschitz constant $\kappa>0$ with respect to $d$ if 
$\vert \omega(x)-\omega(y)\vert\leq \kappa d(x,y)$ for all $x,y\in X$. Furthermore, for $U\subseteq X$, a function $u:[0,\infty)\times U\rightarrow\mathbb{R}$ is called a \emph{solution to the parabolic equation on $U$} if $t\mapsto u_t(x)$ is continuous on $[0,\infty)$, differentiable on $(0,\infty)$ and satisfies
\begin{equation*}
    -\mathcal{H} u_t(x) = u_t'(x)
\end{equation*}
for all $t>0$ and $x\in U$, where $u_t$ is extended by $0$ to $X\setminus U$ and $u_t'$ denotes the time derivative of $u_t.$

\begin{lemma}\label{impfinite}
    Let $K\subseteq X$ be a finite subset. Let $ \omega$ be a bounded real valued Lipschitz function with Lipschitz constant $\kappa>0$ and  let $u$ be a solution to the parabolic equation on $K$. Then,
    \[t\mapsto \exp{\left(2\lambda_2(K)t-\frac{2t}{s^2}(\cosh{\frac{\kappa s}{2}}-1)\right)}E_u(t)\]
is nonincreasing  in $t\geq 0,$ where $\lambda_2(K)$ is the smallest eigenvalue of $H_K$ and \[E_u(t) = \sum_{ K}m|u_t|^2 \mathrm{e}^{\omega}.\]
\end{lemma}
\begin{proof}
We follow the proof of \cite[Lemma~3.2]{BHY2} closely. Write $(f\otimes g)(x,y)=f(x)g(y)$, $x,y\in X$.
We take the time derivative of $t\mapsto E_u(t)$ and use the parabolic equation, as well as  Green's formula to obtain
\begin{multline*}
 \frac{d}{dt}E_u(t)   = 2\sum_{K}m\mathrm{Re}( \overline{u_t} u_t' )\mathrm{e}^{\omega} = -2\sum_{K}m\mathrm{Re}( \overline{u_t}H_K u_t)\mathrm{e}^{\omega} 
 =-2 \mathrm{Re}\, Q(u_t,\mathrm{e}^\omega u_t)
 \\ = -2Q(u_t \mathrm{e}^{{\omega}/{2}}) - 2 \sum_{K\times K}b \mathrm{Re}\left[\overline o\cdot(\mathrm{e}^{\omega\slash 2}u_t\otimes\mathrm{e}^{\omega\slash 2}\overline u_t )\right]\left(1-\cosh\left(\nabla \mathrm{e}^{\omega/2}\right)\right)\\
 \leq -2\lambda_2(K) E_u(t) + \frac{2}{s^2}(\cosh{\frac{\kappa s}{2}}-1)E_u(t),
\end{multline*}
where in the last step we used the Lipschitz property of $\omega$ and that $d$ is an intrinsic metric with jump size $s$ (cf. \cite[Lemma~3.2]{BHY2} for details of the specific estimates). This implies the lemma.
\end{proof}

With this preparation, we can now prove Theorem \ref{dgg}.

\begin{proof}[Proof of Theorem \ref{dgg}]
We will only prove that the function

\[t\mapsto\exp{\left(2\lambda_2t-\frac{2t}{s^2}(\cosh{\frac{\kappa s}{2}-1})\right)}E(t),\qquad t\geq 0,\]
is non-increasing, where \[E(t) = \Vert\mathrm{e}^{\omega\slash 2}S_2(t)u\Vert^2_2\]
for any initial condition $u\in\ell^2$ and $\omega$ any real Lipschitz function with Lipschitz constant $\kappa>0.$ Once we have achieved that, the rest of the proof follows exactly as in \cite[Theorem~1.2]{HL} or \cite[Theorem~1.1]{BHY2}.\\
First, take an exhaustion $K_n\subseteq X$, $n\in\mathbb{N}$, of finite sets. Because \[\lambda_2 = \inf_{\varphi\in C_c( X),\Vert\varphi\Vert=1}{Q(\varphi)} \] it easily follows that $\lambda_2 = \lim_{n\to\infty}\lambda_{2,n}$ with $\lambda_{2,n}=\lambda_2(K_n)$ being the smallest eigenvalue of $H_{K_n}$. Furthermore, by Proposition \ref{semiconv} we have that 
    \[ u_{n,t} = \iota_{K_n}\mathrm{e}^{-tH_{K_n}}\iota_{K_n}^*u\rightarrow S_2(t)u\quad\text{in}~\ell^2.\] Since $\mathrm{e}^\omega$ is bounded, we conclude that
\[\exp{\left(2\lambda_2 t-\frac{2t}{s^2}(\cosh{\frac{\kappa s}{2}}-1)\right)}E(t) = \lim_{n\to\infty}\exp{\left(2\lambda_{2,n} t-\frac{2t}{s^2}(\cosh{\frac{\kappa s}{2}}-1)\right)}E_{u_n}(t).\]
The desired monotonicity now follows from Lemma \ref{impfinite}.
\eat{We assume that $d(A,B)<\infty$ because otherwise the statement simply follows from the Cauchy-Schwarz inequality. The rest of the proof   now carries over verbatim from \cite[Theorem~1.2]{HuaLuo} or \cite[Theorem~1.1]{BauerHuaYau}.    }
\end{proof}

The following corollary is an immediate consequence of Theorem \ref{dgg} and provides explicit heat kernel estimates. It is inspired by \cite[Lemma~3.7]{BauerHuaKeller} but with a sharper constant $C_\beta$. Specifically, the constant $C_\beta$ becomes arbitrarily small for small enough $\beta>0.$ This will be useful in the upcoming section.

\begin{corollary}\label{DGGCorollary}
  For all $\beta>0$ there exists a constant $C_\beta = s^{-2}(\cosh(\beta s)-1)$ such that for all $t\geq 0$, $x,y\in X$,
    \[\vert p_t(x,y)\vert\leq \frac{1}{\sqrt{m(x)m(y)}}\mathrm{e}^{-\beta d(x,y)+C_\beta t-\lambda_2t}.\]
\end{corollary}
\begin{proof}
    Consider the function \[f(r) = s^{-2}\left(rs\mathrm{arcsinh}(sr)-\sqrt{1+r^2s^2}+1\right)-\beta r\]
for $r\geq 0.$ Direct calculation shows that $f'(r) = s^{-1}\mathrm{arcsinh}(rs)-\beta$ and $f''(r)>0.$ Consequently, $f$ achieves its minimum in $r_0 = \sinh{(s\beta)}\slash s.$ Thus,
\[f(r)\geq f(r_0) = \frac{1-\cosh(\beta s)}{s^2}.\]
The statement follows from Theorem \ref{dgg} after observing, that $(t/s^2)\zeta(sd( x, y)/t) = t(f(d( x, y)\slash t)+\beta d( x, y)\slash t).$
\end{proof}

\section{Magnetic Schrödinger Semigroups on $\ell^p$}\label{s:semigroupsonlp}
In this section we want to investigate conditions under which the semigroup $S_2(t)=\mathrm{e}^{-tH_2}$ consistently extends to strongly continuous semigroups $S_p$ on $\ell^p$ for $p\in[1,\infty].$ Here, consistently means that $S_2(t) = S_p(t)$ on $\ell^2\cap\ell^p$ for all $t\geq 0.$ For graph Laplacians, which arise from Dirichlet forms, this is a well-known consequence of the Markov property \cite{DaviesLinOp, KLW}. However, for magnetic Schrödinger operators, this is in general not clear.\\
Let us recall some general facts about semigroups obtained in this way. Let $1\leq p<\infty$ and $-H_2$ and $-H_p$ be the generators of $S_2$ and $S_p$ respectively. If $p>1,$ then we also obtain  the strongly continuous semigroup $S_q$ on $\ell^q$ for $q$ being the H\"older dual of $p.$ This follows from the symmetry of $S_2,$ duality of $\ell^p$ and $\ell^q$ and density of $\ell^p\cap\ell^q$ in $\ell^q.$ Moreover, one has for all $t\geq 0$ that $S_p(t) = (S_q(t))^*$ and $H_p = (H_q)^*$ where $\ast$ indicates the Banach space adjoint of an operator. This also implies the relation of the spectra $\sigma(H_p) = \sigma(H_q)^*,$ where $\ast$ indicates complex conjugation, i.e., the spectra get flipped along the real axis. Whenever $S_1$ exists, we direcly define $S_\infty(t) = (S_1(t))^* $ on $\ell^\infty$ and $H_\infty = (H_1)^*.$ Here, we should note that $S_\infty$ will in general not be strongly continuous but only weak-$\ast$-continuous due to the fact that $\ell^1\cap\ell^\infty$ is not dense in $\ell^\infty.$ The semigroups turn out to be mutually consistent, i.e., for all $1\leq p,q\leq\infty$ one has $S_p = S_q$ on $\ell^p\cap\ell^q.$

We present two criteria that ensure  that $S_2$ consistently extends to $\ell^p.$ The first criterion, which can be viewed as a curvature condition in case of the Hodge Laplacian, is to assume a form bound on the negative part of the potential $c$. The second one is an assumption on the volume growth of distance balls with respect to an intrinsic metric $d$ of finite jump size.\\
Again, throughout this section, we assume that $(b,o,c)$ is a magnetic Schrödinger graph over $(X,m)$ such that $Q$ is closable in $\ell^2$ and positive.

\subsection{The Case of Form Bounded Potentials}

The following result is the main theorem of this subsection. Recall that $\lambda_2$ denotes the bottom of the $\ell^2$-spectrum of $ H_2$ and $c$ is said to be form bounded with bounds $M>0$ and $C\geq 0$ if $ \| \sqrt{c_-/m}\varphi\|_2^2\leq M Q(\varphi)+C\Vert\varphi\Vert^2_2$ for all $\varphi\in C_c(X)$. We denote the operator norm of an operator $T:\ell^p\rightarrow\ell^q$ by $\Vert T\Vert_{p,q}$ for $p,q\in[1,\infty].$

\begin{theorem}[Form bounded potentials]\label{formboundedpotential} Assume $c$ is form bounded with bounds $M>0$ and $C\geq 0$.
    For all $p\in I = [2\frac{M+1}{M}-2\frac{\sqrt{M+1}}{M},2\frac{M+1}{M}+2\frac{\sqrt{M+1}}{M}]$ and all $t\geq0$, one has that $S_2(t)$ is bounded as an operator on $\ell^p$ and \[\Vert S_2(t)\Vert_{p,p}\leq \mathrm{e}^{-D_pt},\]  where $D_p =\left(C_p-M(1-C_p)\right)\lambda_2-C(1-C_p),$ $C_p = 4\frac{p-1}{p^2}$.  In particular, $S_2$ extends consistently to strongly continuous semigroups $S_p$ on $\ell^p.$
\end{theorem}

Theorem \ref{formboundedpotential} is inspired by \cite[Proposition 3.3]{Magniez}, which itself relies on ideas from \cite{LiSe}.
Observe that $C_p = 4\frac{p-1}{p^2}\in [0,1]$ and, for $p\in I$, we have  $C_p-M(1-C_p)\geq 0$. 
 For the proof, let us denote the positive part of the  quadratic form $Q$ on $C_c(X)$ by
 \[q = {Q}+\frac{c_-}{m}.\]

\begin{lemma}\label{pFormEstimate}
    Let $  f\in C_c(X)$ and 
    $p\in [1,\infty)$. Then,
     \[\mathrm{Re}\, q(f,f\vert  f\vert^{p-2})\geq 4\frac{p-1}{p^2} q( f\vert  f\vert^{(p-2)\slash 2}).\]
\end{lemma}
\begin{proof}
    We deduce the statement from the following pointwise estimate.
    Fix $x,y\in X$ and let $\vert  f(x)\vert\geq\vert  f(y)\vert.$ Assume that $  f(x)\neq 0.$ Let $C_p = 4\frac{p-1}{p^2}.$ Note that $1-C_p = (\frac{p-2}{p})^2.$ Recall that $\nabla_{o} f=f(x)-o(x,y)f(y)$ and $\nabla f=f(x)-f(y)$. Then
\begin{align*}
\mathrm{Re}(\nabla_{o} f& \overline{\nabla_{o}(f\vert  f\vert^{p-2})})   - C_p\left|\nabla_{o}( f \vert  f\vert^{\frac{p-2}{2}})\right|^2\\
    =& \,(1-C_p)\left(\vert  f(x)\vert^p+\vert  f(y)\vert^p\right)
    \\& -\mathrm{Re}\left(o(x,y) f(y)\overline{f(x)}\right)\left(\vert f(x)\vert^{p-2}+\vert f(y)\vert^{p-2}-2C_p\vert f(x)f(y)\vert^{\frac{p-2}{2}} \right)\\
\geq&\, (1-C_p)\left(\vert  f(x)\vert^p+\vert  f(y)\vert^p\right)
    \\& - \left|  f(x)  f(y)\right|\left((\nabla |  f|^{\frac{p-2}{2}})^2 + 2(1-C_p)\vert  f(x)  f(y)\vert^{\frac{p-2}{2}} \right)\\
= &\,(1-C_p)\left(\nabla \vert  f \vert^{\frac{p}{2}} \right)^2-\vert  f(x)\vert\vert  f(y)\vert\left(\nabla \vert  f \vert^{\frac{p-2}{2}} \right)^2\\
=&\,  \vert  f(x)\vert^p\left(\left(\frac{p-2}{p}\right)^2\left(1-\left(\frac{\vert  f(y)\vert}{\vert  f(x)\vert}\right)^\frac{p}{2}\right)^2-\frac{\vert  f(y)\vert}{\vert  f(x)\vert}\left(1-\left(\frac{\vert  f(y)\vert}{\vert  f(x)\vert}\right)^{\frac{p-2}{2}}\right)^2\right)\\
\geq &\,0,
\end{align*}
where we conclude the positivity in the last step from the inequality
    \[\left(\frac{p-2}{p}\right)^2(1-t^{p\slash 2})^2\geq t(1-t^{\frac{p-2}{2}})^2\]
 for all $t\in [0,1]$, $p\in [1,\infty),$ which follows from \cite[Lemma 2.1, inequality~(2.4)]{LiSe}.
 
 The case $  f(x) = 0$ is trivial. Multiplying with $b(x,y)$, adding the positive part of the potential on both sides and summing over $x,y$ yields the statement.
\end{proof}

\begin{remark}
    One may prove in a very similar fashion to Lemma \ref{pFormEstimate} the estimate
    \[\kappa(p)q( f\vert  f\vert^{(p-2)\slash 2})\geq\mathrm{Re}\, q(f,f\vert  f\vert^{p-2})\]
    for all $f\in C_c(X)$ and $p\in [1,\infty),$ where \[\kappa(p) = \sup_{t\in(0,1)}\frac{(1+t)(1+t^{p-1})}{(1+t^{p\slash 2})^2}.\] The proof uses the estimate \[t(1-t^{\frac{p-2}{2}})^2\leq(\kappa(p)-1)(1+t^{p\slash 2})^2\] for all $t\in [0,1],$ which one deduces from \cite[Lemma 2.1, inequality~(2.5)]{LiSe}.
\end{remark}

We also note the following extension lemma for strongly continuous semigroups on $\ell^2$ which is of independent interest.

\begin{lemma}\label{strongcont}
    Let $ (T({t})) $ be a strongly continuous semigroup on $ \ell^{2} $ such that \[ \|T({t})f\|_{p} \leq e^{-At}\|f\|_{p} \] for all $ f \in C_c(X) $ and some $ A \in \mathbb{R} $ and some $ 1\leq p <\infty $. Then $ (T({t})) $ extends to a strongly continuous semigroup on $ \ell^{p} $ with the same bound.
\end{lemma}
\begin{proof}
By density of $C_c(X)$ in $\ell^p$ we can extend $T(t)$ for all $t\geq 0$ to a bounded operator $T_p(t)$ on $\ell^p.$ The family of operators $(T_p(t))$ then satisfies the semigroup property $T_p(s+t) = T_p(s)T_p(t)$ trivially, as it holds on the dense subspace $C_c(X).$ 

For the strong continuity, it is sufficient to show that      $\lim_{t\rightarrow0^+}\Vert T_p(t)  f-  f\Vert_p=0$ for all $  f\in C_c(X)$ since $C_c(X)$ is dense in $\ell^p$ and we already know that $T_p$ is exponentially bounded. Let $  f\in C_c(X)$ and $K= \mathrm{supp}\,  f.$ As $T_p = T$ on $C_c(X)$ and $T$ is strongly continuous, we conclude from finiteness of $K$ and pointwise convergence  that 
 \[\lim_{t\rightarrow 0^+}\sum_K m\vert T_p(t)  f-f\vert^{p}=\sum_K m\lim_{t\rightarrow 0^+}\vert T(t)  f-f\vert^{p}=0 .\] 
By the same reasoning, we also have that
    \(\sum_K m\vert T_p(t)  f\vert^{p}\to \sum_K m\vert  f\vert^p \) as ${t\rightarrow 0^+}$. So, for the remaining part of the sum, we use the triangle inequality and the bound on $T_p(t)$ to get 
    \begin{align*}
        \sum_{X\setminus K}m\vert T_p(t)  f-  f\vert^p & = \sum_{X\setminus K  }m\vert T(t)  f\vert^p 
        =\Vert T(t)  f\Vert_p^{p} - \Vert   f\Vert_p^{p}  -\! \sum_{K   }m(\vert T_p(t)  f\vert^p-\vert   f\vert^p) 
        \\
        &\leq (\mathrm{e}^{-At}-1)\Vert  f\Vert_p^{p} -\sum_{K}m(\vert T_p(t)  f\vert^p-\vert   f\vert^p) .
    \end{align*}
By the argument above, the right hand side tends to zero for  $t\to0.$   
\end{proof}

We are now in the position to prove Theorem \ref{formboundedpotential}.

\begin{proof}[Proof of Theorem \ref{formboundedpotential}]
    Let $  f \in C_c(X)$ and $\mathrm{supp}\,  f\subseteq K$ for some finite set $K\subseteq X$ and let $  f_t = \mathrm{e}^{-tH_K}  f,$  $t\geq 0$, with $H_K$ defined above Proposition~\ref{semiconv}. Note that $  f_t' = -H_K  f_t.$
    We identify $  f_t\in\ell^p(K,m)$ as elements in $\ell^p$ via extension by zero to $ X.$ Then, for all $p\in I =  [2\frac{M+1}{M}-2\frac{\sqrt{M+1}}{M},2\frac{M+1}{M}+2\frac{\sqrt{M+1}}{M}],$ we have $ p\neq 1,\infty  $ since $ M>0 $ and we obtain
    \begin{multline*}
        -\frac{1}{p}\frac{d}{dt}\Vert  f_t\Vert_p^p =\sum_{K}m\vert  f_t\vert^{p-2}\mathrm{Re}(\bar f_t H_K f_t)
         =\mathrm{Re} \,{Q}(f_t,f_t\vert  f_t\vert^{p-2})\\
        \geq \left[\left(C_p-M(1-C_p)\right)\lambda_2-C(1-C_p)\right]\Vert  f_t\Vert_p^p = D_p\Vert  f_t\Vert_p^p,
    \end{multline*}
    where $D_p =\left[\left(C_p-M(1-C_p)\right)\lambda_2-C(1-C_p)\right]$, $C_p = 4\frac{p-1}{p^2}$ and $\lambda_2$ is the bottom of the $\ell^2$-spectrum of the operator $ H_2.$
Here, the inequality follows from Lemma~\ref{pFormEstimate}, the form bound of $c_-$ and the fact that 
 $\left(C_p-M(1-C_p)\right)\geq 0$ for all $p\in I.$

 What we have shown is equivalent to 
 \[\frac{d}{dt}\ln({\Vert  f_t\Vert}_p^p)\leq -{p}D_p.\] Integrating both sides from $0$ to $t$ and taking exponentials yields
 \[\Vert  f_t\Vert^p_p\leq \mathrm{e}^{-tpD_p}\Vert  f\Vert^p_p.\] Now let $K_n$, $n\geq 1$, be an exhaustion of $ X$ with finite subsets. We infer from Proposition \ref{semiconv} that $\mathrm{e}^{-t{H}_{K_n}}  f$ converges pointwise for all $t\geq 0$ to $S_2(t)  f$ as $n\rightarrow\infty.$ Thus, by Fatou's lemma we get 
 \[\Vert S_2(t)  f\Vert_p^p\leq \liminf_{n\to \infty}\Vert \mathrm{e}^{-t{H}_{K_n}}  f\Vert_p^p\leq\mathrm{e}^{-tpD_p}\Vert  f\Vert_p^p.\]
 The fact that $S_2$ extends to a strongly continuous semigroup on $\ell^p$ follows from density of $C_c(X)$ in $\ell^p$ and  Lemma~\ref{strongcont}.
 \end{proof}

Theorem \ref{formboundedpotential} excludes the case where $p = 1.$ However, using Kato's inequality, we can show that the statement still holds under the stronger assumption of a lower bounded potential, which corresponds to a form bound $M=0.$

\begin{theorem}[Bounded potentials]\label{boundedpotential}
    Suppose the potential $c\slash m$ is bounded from below, i.e., there exists $K\in\mathbb{R}$ such that $c\slash m\geq -K.$ Then, for all $p\in[1,\infty]$ and $t\ge 0$,
        \[\Vert S_2(t)   \Vert_{p,p}\leq \mathrm{e}^{Kt} .\]
In particular, $S_2$ extends to a strongly continuous semigroup for $1\le p<\infty$ and a weak-$\ast$-continuous semigroup for $p=\infty.$
\end{theorem}
\begin{proof}
    Under the given assumptions, we may apply \cite[Theorem~5.2 (a)]{GKS} to obtain for all $x,y\in X$,    $t\geq 0$, the estimate 

    \[\vert p_t(x,y)\vert\leq \mathrm{e}^{Kt}\tilde p_t(x,y),\]
    where $\tilde p_t$ is the kernel of the semigroup on $\ell^2$ generated by $-\Delta,$ the Laplacian associated to the graph $(b,1,0)$ over $(X,m)$. However, $- \Delta$ generates a Markov semigroup \cite[Proposition~2.6]{KLW} and so for all $x\in X$
    \[\sum_{y\in X} m(y)\tilde p_t(x,y)\leq 1.\] Thus, for $ f\in C_c(X)$,
    \[\Vert S_2(t) f\Vert_1\leq\mathrm{e}^{Kt}\sum_{y\in X} m(y)\vert f(y)\vert\sum_{x\in X} m(x)\tilde p_t(x,y)\leq \mathrm{e}^{Kt}\Vert f\Vert_1.\] By density of $C_c(X)$ in $\ell^1$, we get that $S_2$ extends consistently to a semigroup $S_1$ on $\ell^1.$ 
The statement follows now for $p = 1$ from Lemma~\ref{strongcont}. By duality, we obtain a weak-$\ast$-continuous semigroup $S_\infty$ on $\ell^\infty$ with the same bound which is consistent with $S_2$ on $\ell^1\cap \ell^\infty$. Interpolation then yields the statement for all $p\in[1,\infty].$
\end{proof}

If the graph $b$ is locally finite, then the statement of Theorem \ref{boundedpotential} becomes, in fact, an equivalence. The argument for this is essentially the one from \cite[Theorem~2.1]{JM}. In particular, for locally finite simplicial complexes, this characterizes lower boundedness of the Forman curvature in terms of an exponential bound of the semigroup generated by the Hodge Laplacian.

\begin{proposition}\label{lowerboundpotential}
    Assume that $b$ is locally finite and that $S_2$ extends consistently to a strongly continuous semigroup $S_1$ on $\ell^1.$ Assume there is $K\in\mathbb{R}$ such that, for all $\omega\in\ell^1$ and $t\geq 0$, we have
    \[\Vert S_1(t)\omega\Vert_1\leq \mathrm{e}^{tK}\Vert\omega\Vert_1.\] Then $c\slash m$ is uniformly bounded from below, i.e., for all $x\in\Sigma$, we have
    
    \[\frac{c(x)}{m(x)}\geq -K.\]
\end{proposition}
\begin{proof}
    We fix $x\in X$ and define the function $\varphi_x$ as 

    \begin{equation*}
        \varphi_x(y) = \begin{cases}
            1, &\text{if}\quad y = x.\\
           \overline{ o(x,y)}, & \text{if}\quad b(x,y)>0. \\
            0,&\text{else}.
        \end{cases}
    \end{equation*}
    Local finiteness implies $\varphi_x\in C_c(X)\subseteq D(H_2)$ due to \cite[Lemma~2.11]{GKS} and Green's formula, \cite[Lemma~2.1]{GKS}. Moreover, we get by construction that \[H_2\varphi_x(x) = \frac{c(x)}{m(x)}.\]
    Then, by consistency of $S_1$ and $S_2$ on $C_c(X),$ we have
    
        \[S_2(t)\varphi_x(x) = \langle S_2(t)\varphi_x,1_x\slash m\rangle =  \langle \varphi_x,S_1(t)1_x\slash m\rangle\] and so \[\mathrm{Re}(S_2(t)\varphi_x(x))\leq \mathrm{e}^{tK}\Vert\varphi_x\Vert_\infty\Vert1_x\slash m\Vert_1=\mathrm{e}^{tK}\]
        as $\Vert\varphi_x\Vert_\infty\Vert1_x\slash m\Vert_1=1$.
    Since $\varphi_x\in D(H_2)$, we get 
    \begin{align*}
        -\frac{c(x)}{m(x)} & = -H_2\varphi_x(x) = \lim_{t\rightarrow 0^+}\frac{\mathrm{Re}(S_2(t)\varphi_x(x))-\varphi_x(x)}{t}\leq \lim_{t\rightarrow 0^+}\frac{\mathrm{e}^{tK}-1}{t} = K.
    \end{align*}This finishes the proof.
\end{proof}

\subsection{The Case of Exponential Volume Growth}\label{sectionVolGrowth}
Let $b$ be a graph over $X$ and $d$ be an intrinsic metric with finite jump size.
We say that $b$ has (uniform) \textit{exponential volume growth} with \textit{growth rate} $\nu>0$, if there is $C\ge 0$ such that for all $x\in X$,  $r>0$ \[m(B_r(x))\leq C\mathrm{e}^{\nu r}m(x),\] where $B_r(x)$ is the metric ball with radius $r$ centered at $x$ with respect to $d.$   We say that $b$ has (uniform) \textit{subexponential volume growth}, if for all $\varepsilon>0$ there is $C_\varepsilon>0$ such that for all $x\in X$,    $r>0$
\[m(B_r(x))<C_\varepsilon\mathrm{e}^{\varepsilon r}m(x).\]

The main result of this subsection is the following theorem.
\begin{theorem}[Exponential volume growth]\label{lpextend}
    Assume the graph   has exponential volume growth with growth rate $\nu>0$. Then,    for all $\beta>3\nu\slash 2$ there exist constants
     $C_\beta = s^{-2}(\cosh(\beta s)-1)$ and $M=M_\beta>0$ such that for all $t\geq 0$  
\begin{align*}
        \Vert S_2(t) \Vert_{1,1} \leq M\mathrm{e}^{(C_\beta-\lambda_2) t}.
    \end{align*}
In particular,
    the semigroup $S_2$ extends consistently to a strongly continuous semigroup  on $\ell^p$ for $1\leq p<\infty$ and to a weak-$\ast$-continuous semigroup on $\ell^\infty.$ 
\end{theorem}

We start with the observation that exponential volume growth implies a bounded combinatorial degree and a bound on the number of vertices in a ball.

\begin{lemma}[Exponential volume growth implies bounded degree]\label{LemmaCombVolGrowth}
    Assume the graph   has exponential volume growth with growth rate $\nu>0$. Then, there is $C>0$ such that for all $x\in X$,    $r>0$
 \[\# B_r(x)  \leq C^2\mathrm{e}^{2\nu r}.\]
    In particular, the combinatorial vertex degree is bounded by $C^2\mathrm{e}^{2\nu s}$, where $s>0$ is the jump size of $d.$
\end{lemma}
\begin{proof}
    First, observe that exponential volume growth allows to estimate $m(x)$ by $m(y)$ for $x,y\in X.$ More precisely,
    \[m(x)\leq m(B_{d(x,y)}(y))\leq C\mathrm{e}^{\nu d(x,y)}m(y).\]
    Next, we observe for $r>0$
    \[\# B_r(x) = \sum_{y\in B_r(x)}\frac{m(y)}{m(y)}\leq C\mathrm{e}^{\nu r}\frac{m(B_r(x))}{m(x)}\leq C^2\mathrm{e}^{2\nu r}.\]
    The statement about the combinatorial vertex degree follows by noting that neighbors of a vertex $x$ are contained in $B_s(x)$.
\end{proof}

\begin{lemma}\label{LemmaVolGrowth}
    Assume the graph  has exponential volume growth with growth rate $\nu>0$. For all  $\beta_0>\frac{3}{2}\nu$  there is $C=C_{\beta_0}>0$ such that, for all $\beta\geq\beta_0$,
    \[\sup_{x\in X}\sum_{y\in X}\sqrt{\frac{m(y)}{m(x)}}\mathrm{e}^{-\beta d(x,y)}<C.\]
    In particular, if the graph has subexponential volume growth, then such a bound holds for all $\beta>0$ with $C=C_\beta>0$.
\end{lemma}

\begin{proof}
 By the Cauchy-Schwarz inequality and employing the Lemma~\ref{LemmaCombVolGrowth}, we get for arbitrary $x\in X$ and $\beta>\frac{3}{2}\nu$
    \begin{align*}
        \sum_{y\in X}\sqrt{\frac{m(y)}{m(x)}}\mathrm{e}^{-\beta d(x,y)}
        &= \frac{1}{\sqrt{m(x)}}\sum_{r= 0}^\infty 
        \sum_{ r\leq d(x,y)<r+1}\sqrt{ {m(y)} }\mathrm{e}^{-\beta d(x,y)}\\
        &\leq \frac{1}{\sqrt{m(x)}}\sum_{r= 0}^\infty m(B_{r+1}(x))^{1\slash 2}\left(\sum_{r\leq d(x,y)<r+1}\mathrm{e}^{-2\beta d(x,y)}\right)^{1\slash 2}\\
        &\leq C\sum_{r = 1}^\infty\mathrm{e}^{\nu r\slash 2-\beta r}  \#B_{r}(x)^{1\slash 2} \leq C'\sum_{r = 1}^\infty \mathrm{e}^{3\nu r\slash 2-\beta r} <\infty.
\end{align*}
The ``in particular'' statement follows by noting that, for subexponential volume growth, we can choose $\nu$ arbitrarily small.
\end{proof}

The lemma above allows us to extend the semigroup $S_2$ to $\ell^p.$

\begin{proof}[Proof of Theorem \ref{lpextend}]
    Let $ f\in C_c(X)$ and $t\geq 0.$ Then, choose $\beta>3\nu\slash 2$ so that
    \begin{align*}
        \sum_ X m\vert S_2(t) f\vert &= \sum_{x\in X} m(x)\Big\vert\sum_{y\in X}m(y)p_t(x,y) f(y)\Big\vert\\
        &\leq \mathrm{e}^{(C_\beta-\lambda_2) t}\sum_{y\in X} m(y)\vert f(y)\vert\sum_{x\in X}\sqrt{\frac{m(x)}{m(y)}}\mathrm{e}^{-\beta d(x,y)}\leq C\mathrm{e}^{(C_\beta-\lambda_2) t}\Vert f\Vert_1.
    \end{align*}
Here, we used the triangle inequality, Corollary \ref{DGGCorollary} and Lemma \ref{LemmaVolGrowth}. Thus, $S_2(t)$ extends uniquely to a bounded operator $S_1(t)$ on $\ell^1.$ From the Riesz-Thorin interpolation theorem,  we conclude that $S_2(t)$ also extends uniquely to bounded operators $S_p(t)$ on $\ell^p$ for all $1\leq p\leq 2.$ The fact that $S_p$ is again a semigroup, i.e., for $t,s\geq 0$, the identity $S_p(t)S_p(s) = S_p(t+s)$ holds, follows as $S_2$ is a semigroup and $\ell^2\cap\ell^p$ is dense in $\ell^p$, $1\le p<\infty$. By noting that $S_2(t)$ is self-adjoint, we conclude by a standard duality argument that the same holds for $2\leq p\leq \infty.$  It remains to show that $S_p$ is strongly continuous for $\ell^p$, $1\le p<\infty$. \\
Let $p = 1.$ We established that $S_1$ is exponentially bounded. Thus, as $C_c(X)$ is dense in $\ell^1,$ it is sufficient to show that $\lim_{t\rightarrow0^+}\Vert S_1(t)f-f\Vert_1 = 0$ for all $f\in C_c(X).$ Let $K = \mathrm{supp}\,f$ and fix arbitrary $\varepsilon >0.$ Let 
\[K_1 = \lbrace y\in X\mid \text{ there is}~x\in K~\text{such that}~d(x,y)\leq1\rbrace.\]  
Since we assume exponential volume growth, we find that 
\[m(K_1 )\leq C\mathrm{e}^{\nu}m(K)<\infty.\] Thus, as $S_2$ is strongly continuous, we conclude for $t$ small enough, that 
\[\sum_{K_1}m\vert S_1(t)f-f\vert = \sum_{K_1}m\vert S_2(t)f-f\vert\leq m(K_1)^{1\slash 2}\Vert S_2(t)f-f\Vert_2<\varepsilon\slash 2.\] 
Moreover, by Lemma \ref{LemmaVolGrowth} we may choose  $\beta>3\nu$ large enough, such that 
\[\Vert f\Vert_1\mathrm{e}^{-\beta\slash 2}\sup_{x\in X}\sum_{y\in X}\sqrt{\frac{m(y)}{m(x)}}\mathrm{e}^{-\frac{\beta}{2} d(x,y)}<\varepsilon\slash 3.\]
Finally, let $C_\beta$ be the constant from Corollary \ref{DGGCorollary} and choose $t$ small enough, such that $\mathrm{e}^{C_\beta t}\varepsilon\slash 3<\varepsilon\slash 2.$ Then, with  Corollary \ref{DGGCorollary}, we obtain
\begin{align*}
    \sum_{X\setminus K_1}m\vert S_1(t)f-f\vert &= \sum_{X\setminus K_1} m\vert S_2(t)f\vert\leq \sum_{x\in X\setminus K_1}m(x)\sum_{y\in K}m(y)\vert f(y)\vert\vert p_t(x,y)\vert\\
    &\leq \mathrm{e}^{C_\beta t}\sum_{y\in K}m(y)\vert f(y)\vert\sum_{x\in X\setminus K_1}\sqrt{\frac{m(x)}{m(y)}}\mathrm{e}^{-\beta d(x,y)}\\
    &\leq \mathrm{e}^{C_\beta t}\sum_{y\in K}m(y)\vert f(y)\vert\sum_{x\in X}\sqrt{\frac{m(x)}{m(y)}}\mathrm{e}^{-\frac{\beta}{2} d(x,y)}\mathrm{e}^{-\frac{\beta}{2}}\\
    &\leq \varepsilon\mathrm{e}^{C_\beta t}\slash 3< \varepsilon\slash 2.
\end{align*}
All in all, we find that for $t$ small enough $\Vert S_1(t)f-f\Vert_1<\varepsilon.$ We obtain the strong continuity for $1\leq p\leq 2$ by interpolation and for $2\leq p<\infty$ by a standard duality argument.\end{proof}


\subsection{The Action of the Generator} In this section, we give a criterion when the generator $H_p$ is a rectriction of the formal Schrödinger operator $\mathcal{H}$.
So, we assume that $ -H_p$ is the generator of a strongly continuous semigroup $S_p$ on $\ell^p$ for some $1\leq p<\infty$ that is consistent with the semigroup $S_2$. Moreover, $H_\infty=(H_1)^*$   whenever $S_1$ exists and is strongly continuous. By general theory about semigroups (\cite{EngelNagelShort,KLW}) such generators exist. Clearly, $ H =  H_2.$ Here we briefly discuss the action of $ H_p$ in case when $ X$ is locally finite. 

\begin{theorem}[Action of $ H_p$]\label{actionHp}
Let $p\in [1,\infty]$ and $q$ the respective H\"older dual. Assume $\mathcal{H} C_c(X)\subseteq\ell^q\cap\ell^2,$ which holds in particular if $b$ is locally finite.  Then \[ H_p =   \mathcal{H}\quad\text{on}~ D( H_p)\quad\text{and}\quad C_c(X)\subseteq D(H_q).\]
\end{theorem}
\begin{proof}
    First, observe that $(H_p+\alpha)^{-1} = (H_2+\alpha)^{-1}$ on $\ell^p\cap\ell^2$ and  $(H_q+\alpha)^{-1} = (H_2+\alpha)^{-1}$ on $\ell^q\cap\ell^2$ for large enough $\alpha>0.$ This follows for $p,q<\infty$ by consistency of the semigroups and the Laplace transform, \cite[Theorem 1.10]{EngelNagelShort} and for $p=\infty$ or $q = \infty $ by duality and self-adjointness of $(H_2+\alpha)^{-1}$. Now, let $f\in D(H_p)$ and $\varphi\in C_c(X).$ 
    We have $\mathcal{H}C_c(X)\subseteq \ell^2$ and so it follows from Green's formula, \cite[Lemma 2.1]{GKS} that $C_c(X)\subseteq D(H_2)$ and $ H_2 = \mathcal{H}$ on $D(H_2).$ Thus,
\begin{align*}
     \sum_ X m(H_p+\alpha)f\overline{\varphi} &= \sum_ X m (H_p+\alpha)f\overline{( H_2+\alpha)^{-1}(\mathcal{H}+\alpha)\varphi}
    = \sum_ X m f\overline{(\mathcal{H}+\alpha)\varphi} .
\end{align*}
 It follows that $\varphi\in D(H_p^*)$ which gives the second statement for $q>1$ as $H_p^* = H_q$. The case $q = 1$ follows from the identity \[\varphi = (H_2+\alpha)^{-1}(\mathcal{H}+\alpha)\varphi = (H_1+\alpha)^{-1}(\mathcal{H}+\alpha)\varphi\in D(H_1).\]
Moreover, by choosing $x\in X$ and $\varphi = 1_x\slash m,$ we see that \[H_pf(x) = \sum_ X mf\overline{\mathcal{H}\varphi} = \mathcal{H} f(x),\]
 where the last equality is due to the Green's formula for $\mathcal H$, \cite[Lemma~2.1]{GKS}. The inclusion $\ell^p\subseteq F$ follows from $\mathcal{H}C_c(X)\subseteq \ell^q$ and Hölder's inequality.
\end{proof}

\eat{
\begin{remark}
We mention that Theorem \ref{actionHp} applies to Laplacians $\Delta^\pm$, $\Delta^H$ on simplicial complexes (see Section \ref{s:weighted}) as well. In fact, in \cite{BK} it was shown that those Laplacians on $\ell^2$ act like Schrödinger operators $\mathcal{H},$ which in turn act as compositions of boundary and coboundary operators $\partial$ and $\delta.$ Theorem \ref{actionHp} then tells us that this is true for the $\ell^p$-generator as well for all $p\in[1,\infty].$

\end{remark}}

\subsection{Analyticity and the Parabolic Equation}
Here, we briefly discuss the analyticity of the semigroup $S_p$ on $\ell^p$ and the solution of the parabolic equation associated to $ H_p$
\begin{equation}\tag{PE} \label{PE}
    - H_p\omega_t(x) = \frac{d}{dt}\omega_t(x),
\end{equation}
for all $t>0$, $x\in X$ and initial conditions $\omega_0\in\ell^p.$ In the previous section, we discussed when $H_p$ is a restriction of the formal Schrödinger operator $\mathcal{H}$ and in this case \eqref{PE} is indeed the parabolic equation associated to $\mathcal{H}$.

While $S_p$ itself does not need to be analytic in the strict sense of \cite[Chapter~II Definition~4.5]{EngelNagelShort}, it is after a spectral shift due to the exponential bound on $S_p$. In other words, we find constants $A\in\mathbb{R}$ such that the semigroup $T(t) = \mathrm{e}^{tA}S_p(t)$ is analytic and bounded. Thus, $T$ maps $\ell^p$ into the domain of its generator $H_p-A$ according to \cite[Chapter~II Theorem~4.6]{EngelNagelShort}. Consequently, $T$ solves the \emph{modified parabolic equation} \[- (H_p-A)\omega_t(x) = \frac{d}{dt}\omega_t(x)\] for all $t> 0,~x\in X$ and initial conditions $\omega_0\in\ell^p$ which is true iff $S_p$ solves (\ref{PE}) for all $\omega_0\in\ell^p.$
\begin{proposition}\label{analyticity}
    \begin{itemize}
        \item [(a)] If the potential is form
bounded with constant $M>0$, then $\omega_t = S_p(t)\omega$ solves 
\eqref{PE} on $X$ for all initial conditions $\omega\in\ell^p$ with \[p\in \left(2\frac{M+1}{M}-2\frac{\sqrt{M+1}}{M},2\frac{M+1}{M}+2\frac{\sqrt{M+1}}{M}\right).\]
If  the potential is  bounded from below, then this is true for all $1<p<\infty.$
        \item [(b)] If the graph   has exponential volume growth, then $\omega_t = S_p(t)\omega$ solves 
        \eqref{PE} on $X$ for all initial conditions $\omega\in\ell^p$ and all $1<p<\infty.$
    \end{itemize}
    \end{proposition}
\begin{proof}
By the spectral theorem the semigroup $S_2$ is analytic and contracting. By Theorem~\ref{formboundedpotential},~\ref{boundedpotential} and~\ref{lpextend}, we get that $\|S_2(t)\|_{p,p}\leq C e^{At}$ for respective $p$ and constants $A,C\ge 0.$ Then the rescaled semigroup $T(t) = \mathrm{e}^{-tA}S_2(t)$ is still contracting and analytic on $\ell^2$ and bounded on $\ell^p.$ A standard application of Stein's interpolation theorem (see \cite[Theorem~1 of Chaper III §2]{Stein} or \cite[Theorem~6.6]{Lunardi}) yields that $T$ extends to an analytic semigroup on $\ell^r$ for all $r$ strictly between $2$ and $p.$ Thus, $T$ solves the modified parabolic equation from above for all initial conditions in $\ell^r$ and, therefore, $S_r$ solves \eqref{PE} for all initial conditions in $\ell^r.$ 
\eat{Stein's interpolation theorem gives analyticity for contraction semigroups \cite[Theorem 1.4.2?]{Daviesbook} whenever $C\le 1$ and $A\le 0$. It is well known that this result can be extended to arbitrary constants  $C\ge 0,A\in\mathbb{R}$ by the discussion of \cite[Section~7?]{Ahrendt}. 
Now solvability of the parabolic equation follows as analytic semigroups map $\ell^p$ into the domain of their generator \cite[4.6 Theorem~(c)]{EngelNagel-oneparameter}. Note that in \cite{EngelNagel-oneparameter} states the result for uniformly bounded semigroups which can however be mitigated by normalizing $S_p$ via $T_p=e^{At}S_p$ cf.~\cite[Section~7?]{Ahrendt}.}
\end{proof}

\eat{
\begin{remark}
    We make no effort here to optimize the specific angle of analyticity of the respective semigroups. Since our arguments rely on Stein's interpolation theorem, one obtains the usual angle $\theta = \pi(1-1\slash p)$ for $1<p<2$ and duality yields angles for $2<p<\infty.$ However, it seems plausible that in the situation of Proposition \ref{analyticity} (a) one could adapt the result from \cite[Theorem 2.1]{LiSe} to obtain a potentially better angle that depends on the constants in the form bound for $c.$
\end{remark}}

\section{Independence of the $\ell^p$-Spectrum}
\label{s:p-independence}

In the previous section, we discussed when the $\ell^2$ semigroup $S_2$ of a magnetic Schrödinger graph $(b,o,c)$ over $(X,m)$ extends to a strongly continuous semigroup $S_p$ on $\ell^p.$  Now, we study the stability of the spectrum $\sigma( H_p)$ under variation of $p.$ 

Assume that $S_2$ extends to a strongly continuous semigroup $S_p$ on $\ell^p$ for some  $1\leq p<\infty$. Then $ S_{p} $ is \emph{exponentially bounded}, i.e., \[\Vert S_p(t)\Vert_p\leq M e^{A t}\] for some $M\geq 1$, $A\in\mathbb{R}$, see \cite[Proposition~1.4]{EngelNagelShort}. 
This gives us the representation of the resolvents via the \emph{Laplace transform}
\[( H_p-\lambda)^{-1}  f= \int_0^\infty\mathrm{e}^{\lambda t}S_p(t) f\mathrm{d}t\] for all $\mathrm{Re}\lambda<-A$ (see \cite[Theorem~1.10]{EngelNagelShort}), where the integral is understood as an improper Riemann integral. In particular, we see that those $\lambda$  are in the resolvent set denoted by $\rho( H_p)=\mathbb{C}\setminus \sigma(H_p).$ Let us discuss this in the case of (sub-)exponential volume growth. 

\begin{example}[(Sub-)exponential volume growth]\label{exampleexpvolgrowth}
 According to the previous section, for graphs with exponential volume growth with exponent $\nu>0$, the semigroup $S_1$ is strongly continuous and exponentially bounded, i.e., for $\beta>3\nu\slash 2$, we have  $\Vert S_1(t)\Vert_{1,1}\leq M_\beta\mathrm{e}^{(C_\beta -\lambda_{2})t}$  for a constant $M_\beta$ and $C_\beta= s^{-2}(\cosh(\beta s)-1)$ as in Corollary~\ref{DGGCorollary}.
  As $S_2$ is a contraction semigroup with $\Vert S_2(t)\Vert_{2,2}\leq \mathrm{e}^{-\lambda_{2}t}$, we conclude from the Riesz-Thorin interpolation theorem and duality arguments that \[\Vert S_p(t)\Vert_p\leq M_\beta^{\frac{2-r}{r}}\mathrm{e}^{\left(\frac{2-r}{r}C_\beta -\lambda_{2}\right) t},\] where $r = \min\lbrace p,p\slash(p-1)\rbrace.$
By the discussion above, we have for all $\mathrm{Re}\lambda<-\frac{2-r}{r}C_\beta +\lambda_{2}$ that $\lambda\in\rho( H_p).$ Specifically, since $C_\beta = s^{-2}(\cosh(\beta s)-1)$ and $\beta>3\nu/2$, we obtain the estimate
\[\inf\sigma( H_2)\leq \inf\mathrm{Re}\,\sigma( H_p) +\frac{2-r}{rs^{2}} (\cosh(3\nu s/2)-1),
 \]
 where $s$ is the jump size of the intrinsic metric $d.$
As we noted in the remark after Theorem \ref{dgg}, we find that the constant $C_\beta$ gets arbitrarily small for small $\beta >0.$ Hence, in case of uniform subexponential volume growth, we may deduce that for all $1\leq p\leq\infty$ and $\varepsilon >0$  there is $M=M_{\varepsilon,p}>0$ such that
\[\Vert S_p(t)\Vert_p\leq M \mathrm{e}^{\varepsilon t}\]
and 
\[\inf\sigma( H_2)\leq \inf\mathrm{Re}\,\sigma( H_p).\] 
This is a first glimpse at the stability of the $\ell^p$-spectrum.
\end{example}

\subsection{Uniformly Positive Measure and Spectral Inclusion}
In this section we discuss the spectral inclusion of the $\ell^2$-spectrum in the $\ell^p$-spectrum in case of a uniformly positive measure, i.e., $\inf_x m(x)>0.$ This assumption implies that $\ell^p\subseteq\ell^q$ for all $1\leq p\leq q\leq\infty.$ 

\begin{theorem}\label{finitemeasure}  Let $ 1\leq p< \infty $. Assume that  $ S_{2} $ extends consistently to a strongly continuous semigroup $S_p$ on $ \ell^{p} $, $p\in[1,\infty)$,  and  \[\inf_{x\in X} m(x)>0.\] Then,
  \[\sigma( H_2)\subseteq\sigma( H_p).\]
\end{theorem}

\begin{proof}
Note that if $S_2$ extends consistently to a strongly continuous semigroup $S_p$ on $\ell^p$, $p\in[1,\infty)$,  by duality we also have that $S_2$ extends to a strongly continuous semigroup $S_q$ on $\ell^q,$ where $q$ is the respective H\"older dual of $p\neq 1$ and to a weak-$\ast$-continuous semigroup on $\ell^\infty$ if $p=1$. In either case, we have $  \sigma( H_p)\cap\mathbb{R} = \sigma( H_q)\cap\mathbb{R}.$ Thus, as the $\ell^2$-spectrum is real, we can assume without loss of generality $1\leq p\leq 2$ and let $q$ be its Hölder dual.  We make  the following claim:\medskip

    \emph{Claim:} $ H_p\subseteq H_2\subseteq H_q.$\medskip
    
If the claim holds, then the resolvents $( H_p-z)^{-1}$ and $( H_q-z)^{-1}$ are consistent for all $z\in \rho( H_p)\cap\rho( H_q)$ because for all $f\in\ell^p\subseteq \ell^q$
\[( H_q-z)( H_p-z)^{-1}f = ( H_p-z)( H_p-z)^{-1}f = f= ( H_q-z)( H_q-z)^{-1}f \]
and since $( H_{q}-z)$ is injective.

Now, take $z\in\rho( H_p)\cap{\rho( H_q)}.$ Then we conclude by consistency of the resolvents and interpolation, that $( H_q-z)^{-1}$ is bounded as an operator on $\ell^2.$ Moreover, it is an inverse of $ H_2-z$ and, therefore, $z\in\rho( H_2)$: To see this, take $f\in D( H_2).$ It follows from the claim that 
\[( H_q-z)^{-1}( H_2-z)f = ( H_q-z)^{-1}( H_q-z)f = f.\]
On the other hand, for $f\in\ell^2$ choose $f_n\in\ell^p,n\geq 0$, such that $f_n\rightarrow f$ in $\ell^2.$ Then $( H_q-z)^{-1}f_n = ( H_p-z)^{-1}f_n\in D( H_p)\subseteq D( H_2)$ and by the claim
\[( H_2-z)( H_q-z)^{-1}f_n = ( H_q-z)( H_q-z)^{-1}f_n = f_n\rightarrow f.\] Since $( H_q-z)^{-1}$ is bounded on $\ell^2$ and $ H_2$ is closed, we conclude that $(H_q-z)^{-1}f\in D(H_2)$ and
\[( H_2-z)( H_q-z)^{-1}f = f.\]
Hence, $\rho(H_p)\cap\rho(H_q)\subseteq\rho(H_2)$ and therefore $\sigma(H_2)\subseteq \sigma(H_p)\cup\sigma(H_q).$ Since $\sigma(H_2)$ is real and $\sigma(H_p) \cap\mathbb{R}= \sigma(H_q)\cap\mathbb{R},$ the assertion follows.\\
It remains to show the claim. By assumption, the semigroup $S_p $ is exponentially bounded and consistent with $S_2$. Thus, we infer that the resolvents $( H_p+\alpha)^{-1}$ and $( H_2+\alpha)^{-1}$ exist and are consistent for large $\alpha>0,$ which is a simple consequence of the Laplace transform. Let $f\in D( H_p).$ Then, as $\ell^p\subseteq \ell^2,$
\[f = ( H_p+\alpha)^{-1}( H_p+\alpha)f = ( H_2+\alpha)^{-1}( H_p+\alpha)f\in D( H_2).\]
Applying $ H_2+\alpha$ to both sides yields
\begin{align*}
    ( H_2+\alpha) f &=   ( H_p+\alpha)f,
\end{align*}
which implies the first inclusion of the claim. The second inclusion follows  by duality: Let $f\in D( H_2).$ Then, for all $g\in D( H_p)\subseteq D( H_2)$, $1\leq p\leq 2$, we have by self-adjointness of $ H_2$
\[\langle   H_2 f, g\rangle = \langle f,  H_2 g\rangle = \langle  f,  H_p g\rangle.\]
Hence, $f\in D(( H_p)^*) = D( H_q)$ and $ H_2 f =  H_q f.$ This finishes the proof of the claim and, thus, of the theorem.
\end{proof} 

\begin{remark}
    If $S_2$ extends to a strongly continuous semigroup on $\ell^1$, then by duality, we also have
    \[\sigma( H_2)\subseteq  \sigma( H_1)\cap\mathbb{R} = \sigma( H_\infty)\cap\mathbb{R} \]
    and, therefore, $\sigma( H_2)\subseteq\sigma( H_\infty)$ as well. 
\end{remark}

The theorem above has the following immediate consequence in case of  subexponential volume growth.

\begin{corollary}
    If the graph has uniform subexponential volume growth   and uniformly positive measure $\inf_xm(x)>0,$ then, for all $1\leq p\leq\infty,$
    \[\inf\sigma( H_2) = \inf\mathrm{Re}\,\sigma( H_p).\]
\end{corollary}
\begin{proof}
    This follows directly from Theorem \ref{finitemeasure} and Example~\ref{exampleexpvolgrowth} prior to it.
\end{proof}

\subsection{Independence of the $\ell^p$-Spectrum}

In this section we study  the independence of the $\ell^p$-spectrum in the case of form bounded potentials and  subexponential volume growth.


\begin{theorem}[Independence of $\ell^p$-spectrum]\label{lpIndependence}
   If the potential is form bounded  and the graph has uniform subexponential volume growth, then, for all $1\leq p\leq\infty$, \[\sigma( H_2) = \sigma( H_p).\] 
\end{theorem}

The proof of the theorem above is similar to the graph case, \cite{BauerHuaKeller}, and the case of manifolds \cite{Sturm}.
The key estimate we need is of the form 
\[\Vert\mathrm{e}^{-\psi}( H_2-z)^{-1}\mathrm{e}^\psi\Vert_{2,2}<C\] on compact subsets of $\rho( H_2),$ where $\psi$ is an arbitrary real bounded Lipschitz function with respect to the metric $d$ and fixed Lipschitz constant $\varepsilon>0.$ But first, we need a technical lemma which is found for Dirichlet forms in \cite[Lemma~3.5]{GHM}.


\begin{lemma}\label{Q0} Assume the potential $c$ is form bounded.
    Let $ f\in D(Q)$ and $\psi$ be a bounded Lipschitz function. Then, $\psi f\in D(Q).$ 
\end{lemma}

\begin{proof}
    Let $ f_n\in C_c(X),n>0$, be a Cauchy sequence with respect to the form norm $\Vert\cdot\Vert_{{Q}}=(Q(\cdot)+\|\cdot\|^2_2)^{1/2}.$ We need to show that $\psi f_n$ is such a Cauchy sequence as well. Obviously, $\psi f_n$ is an $\ell^2$-Cauchy sequence. Thus, all that remains to show is that for $\varepsilon>0$ and all $n,m$ large enough \[Q(\psi( f_n- f_m))<\varepsilon.\]
Let  $\kappa$ be the Lipschitz constant of $\psi$. For arbitrary $ f\in C_c(X)$,
\begin{align*}
   0 &\leq Q(\psi f) \leq \sum_{X\times X}b\vert\nabla_o\psi\vert^2 +\sum_ X (c\vee0)\vert\psi f\vert^2\\
    &= \sum_{x,y\in X}b(x,y)\vert f(x)(\psi(x)-\psi(y))+\psi(y)( f(x)-o(x,y) f(y))\vert^2+\sum_ X(c\vee 0)\vert\psi f\vert^2\\&\leq 2\kappa^2\sum_{x,y\in X}b(x,y)\vert f(x)\vert^2d(x,y)^2+2\Vert\psi\Vert^2_\infty\left(Q( f)-\sum_ X (c\wedge 0)\vert f\vert^2\right)\\&\leq 2(\kappa^2+C\Vert\psi\Vert^2_\infty)\Vert f\Vert_2^2+2 (M+1)\Vert\psi\Vert^2_\infty Q( f).
\end{align*}
Here, we used that $d$ is intrinsic and that $c$ is form bounded with bounds $M,C$. Applying this estimate to $ f= f_n- f_m$ finishes the proof.
\end{proof}

The following lemma is the key estimate for the proof of Theorem \ref{lpIndependence} and a variation of \cite[Lemma 4.1]{BauerHuaKeller} on graphs.

\begin{lemma}\label{resolventbound} Assume  the potential is  form bounded.
    Let $K\subseteq\rho( H_2)$ be compact. Then there are constants $\varepsilon,C>0$ such that, for all $z\in K$ and all bounded real Lipschitz functions $\psi$ 
    with Lipschitz constant $\varepsilon$, 
    one has 
    \[\Vert\mathrm{e}^{\psi}( H_2-z)^{-1}\mathrm{e}^{-\psi}\Vert_{2,2}<C.\]
\end{lemma}
\begin{proof}
    Since $\psi$ is a bounded Lipschitz function, so are $\mathrm{e}^{\pm\psi}.$
    Thus, with Lemma \ref{Q0}, we can define the quadratic form 
    $Q_\psi$ on $D(Q_\psi) = D(Q)$ as
    \[Q_\psi(f,g) = Q(\mathrm{e}^{-\psi }f,\mathrm{e}^\psi g)-Q(f,g).\]

By doing an analogous calculation as in \cite[Lemma 4.1]{BauerHuaKeller} one finds that there is $C=C(s)>0$ such that,  for $\delta>0$ and $0<\varepsilon<1$ arbitrary, 
\[\vert Q_\psi (f)\vert \leq C\varepsilon^2\left(1+\frac{1}{\delta}\right)\Vert f\Vert^2_2+\delta \left(Q(f)-\sum_ X (c\wedge0)\vert f\vert^2\right).\]
Now, since $c$ is form bounded, we see that for $f\in D(Q)$

\[ Q(f) - \sum_X (c \wedge 0) \vert f\vert^2 \leq C (Q(f)+\Vert f\Vert^2_2)\] for some $C>0.$ Combining this with the estimate above yields 

\[\vert Q_\psi(f)\vert\leq C\varepsilon^2\left(1+\frac{1}{\delta}\right)\Vert f\Vert^2_2+C\delta\Vert f\Vert^2_2+C\delta Q(f).\] We conclude from \cite[Theorem VI 3.9]{K} that $Q_\psi+Q$ is closed and sectorial and, therefore, is represented by a sectorial operator $ H_\psi.$ Moreover, by the same reference, we find that for $\delta$ and $\varepsilon$ small enough $K\subseteq\rho( H_\psi)$ and 
\[\Vert ( H_\psi-z)^{-1}\Vert_{2,2}< C\] for some constant $C = C(\varepsilon,\delta, K).$ The fact that $C$ may only depend on $K$ but not on $z$ is due to compactness of $K.$

A direct calculation shows that $ H_\psi = \mathrm{e}^\psi H_2\mathrm{e}^{-\psi}$ on $D( H_\psi) = \mathrm{e}^\psi D( H_2).$
Therefore,\[ \Vert\mathrm{e}^{\psi}( H_2-z)^{-1}\mathrm{e}^{-\psi}\Vert_{2,2}=\Vert (\mathrm{e}^\psi H_2\mathrm{e}^{-\psi}-z)^{-1}\Vert_{2,2}=\Vert  ( H_\psi-z)^{-1}\Vert_{2,2} <C.\]
This finishes the proof.
\end{proof}

Under the assumption of subexponential volume growth, we already established in Example~\ref{exampleexpvolgrowth} that the semigroups on $\ell^p$, $1\leq p\leq \infty$, are exponentially bounded. Therefore, the resolvents $( H_p-\lambda)^{-1}$ can be expressed via the Laplace transform as \[( H_p-\lambda)^{-1} = \int_0^\infty\mathrm{e}^{\lambda t}S_p(t)\mathrm{d}t\] for all $\mathrm{Re}\,\lambda<0$, (where the formula holds strongly for $1\leq p<\infty$ and in a weak-$\ast$-sense for $p=\infty$). Since the semigroups $S_p$ are consistent for $1\leq p< \infty,$ we conclude the same for the resolvents on $\lbrace \lambda\in\mathbb{C}\mid\mathrm{Re}\,\lambda<0\rbrace$ due to density arguments.
In particular, the kernel $g_\alpha$ of the resolvent $G_\alpha = ( H_p-\alpha)^{-1}$, $\alpha<0$, is independent of $1\leq p<\infty.$  

The following bounds for $g_\alpha$ are a variation of \cite[Lemma 4.2]{BauerHuaKeller} on graphs and provide the necessary estimates for the proof of Theorem \ref{lpIndependence}.

\begin{lemma}\label{Galpha} Assume the potential $c$ is form bounded and the graph $b$ has subexponential volume growth. For every $\varepsilon>0$, there are $\alpha<0<C<\infty$ such that, for all bounded Lipschitz functions $\psi$ with Lipschitz constant $\varepsilon$,
    \begin{itemize}
        \item [(a)] $\vert g_\alpha(x,y)\vert\leq C(m(x)m(y))^{-1\slash 2}\mathrm{e}^{-\varepsilon d(x,y)}$ for all $x,y\in X.$ 
        \item[(b)] $\Vert \mathrm{e}^\psi G_\alpha\mathrm{e}^{-\psi}m^{1\slash 2}\Vert_{1,2}\leq C.$  
        \item[(c)]  $\Vert m^{1\slash 2}\mathrm{e}^\psi G_\alpha\mathrm{e}^{-\psi}\Vert_{2,\infty}\leq C.$
    \end{itemize}
\end{lemma}
\begin{proof}
    Assertion (a) follows from Corollary \ref{DGGCorollary} and the Laplace transform. More precisely, choosing $\alpha<-C_\varepsilon$ for $C_\varepsilon$ as in Corollary~\ref{DGGCorollary} (for $\varepsilon=\beta$) we get
    \[\vert g_\alpha(x,y)\vert = \Big\vert\int_0^\infty\mathrm{e}^{\alpha t}p_t(x,y)\mathrm{d}t\Big\vert\leq(m(x)m(y))^{-1\slash 2}\mathrm{e}^{-\varepsilon d(x,y)}\int_0^\infty\mathrm{e}^{(\alpha+C_\varepsilon)t}\mathrm{d}t.\] 
    Assertions (b) and (c) follow verbatim as in \cite[Lemma 4.2]{BauerHuaKeller}.
\end{proof}

With the estimates from Lemma \ref{Galpha} and Lemma \ref{resolventbound}, we are now able to show the following crucial boundedness result for the square of the resolvent on $\ell^p.$ The proof follows again the graph setting \cite[Lemma~4.3]{BauerHuaKeller}.

\begin{lemma}\label{lp2bounded}
   Assume the potential $c$ is form bounded and the graph $b$ has subexponential volume growth. Then $( H_2-z)^{-2}$ extends to a bounded operator on $\ell^p$, $1\leq p\leq\infty$, for all $z\in\rho( H_2).$ Moreover, for every compact $K\subseteq\rho( H_2)$, there exists $C>0$ such that for all $z\in K$
    \[\Vert ( H_2-z)^{-2}\Vert_{p,p}\leq C.\]
\end{lemma}

\begin{proof}
     Let $g_z^{(2)}$ be the kernel of $G_z^2 = ( H_2-z)^{-2}.$ 
     For given compact $K\subseteq\rho( H_2)$, we choose $\varepsilon>0$ 
     according to Lemma~\ref{resolventbound} and, then $\alpha<0$ according to Lemma~\ref{Galpha}. By using the resolvent identity twice, we obtain for $z\in K$ that
    \[G_z^2 = (G_\alpha+(z-\alpha)G_\alpha G_z)(G_\alpha+(z-\alpha)G_zG_\alpha) = G_\alpha(I+(z-\alpha)G_z)^2G_\alpha.\]
    Let $\psi$ be a bounded Lipschitz function with Lipschitz constant $\varepsilon.$ Then
    \[m^{1\slash 2}\mathrm{e}^\psi G_z^2\mathrm{e}^{-\psi}m^{1\slash 2} = (m^{1\slash 2}\mathrm{e}^\psi G_\alpha\mathrm{e}^{-\psi})(I+\mathrm{e}^{\psi\slash 2} G_z\mathrm{e}^{-\psi\slash 2})^2(\mathrm{e}^\psi G_\alpha\mathrm{e}^{-\psi}m^{1\slash 2})\] and, therefore,
    \begin{multline*}
       \Vert m^{1\slash 2}\mathrm{e}^\psi G_z^2\mathrm{e}^{-\psi}m^{1\slash 2}\Vert_{1,\infty}
       \\\leq \Vert m^{1\slash 2}\mathrm{e}^\psi G_\alpha\mathrm{e}^{-\psi}\Vert_{2,\infty}\Vert (I+\mathrm{e}^{\psi\slash 2} G_z\mathrm{e}^{-\psi\slash 2})^2\Vert_{2,2}\Vert\mathrm{e}^\psi G_\alpha\mathrm{e}^{-\psi}m^{1\slash 2}\Vert_{1,2}\le C<\infty,
    \end{multline*}
    by the calculation above and the estimates in Lemma \ref{resolventbound} and Lemma \ref{Galpha} with $C=C(K,\varepsilon,\alpha)$. Put $U = m^{1\slash 2}\mathrm{e}^\psi G_z^2\mathrm{e}^{-\psi}m^{1\slash 2}.$ Then, we have shown that $U$ is a bounded operator from $\ell^1$ to $\ell^\infty$ and as such admits a kernel \[k_U(x,y) = (m(x)m(y))^{1\slash 2}\mathrm{e}^{\psi(x)-\psi(y)}g_z^{(2)}(x,y).\] By the  Hölder inequality, cf. \cite[Lemma 3.5 (b)]{BauerHuaKeller}, we conclude that $ \|k\|_\infty=\|U\|_{1,\infty}\leq C $, so that,    
    for all $x,y\in X$
    \[\vert g_z^{(2)}(x,y)\vert\leq C(m(x)m(y))^{-1\slash 2}\mathrm{e}^{\psi(y)-\psi(x)}.\]
    If we  fix $x,y\in X$ and choose $\psi: v\mapsto \varepsilon (d(v,y)\wedge d(x,y))$, then the estimate becomes
    \[\vert g^{(2)}_z(x,y)\vert\leq C(m(x)m(y))^{-1\slash 2}\mathrm{e}^{-\varepsilon d(x,y)}.\] Let $f\in\ell^1.$ Then, we get
\begin{align*}
    \Vert G^2_zf\Vert_1&\leq\sum_{x\in X }m(x)\vert f(x)\vert\sum_{y\in X }m(y)\vert g^{(2)}_z(x,y)\vert\\&\leq C\sum_{x\in X }m(x)\vert f(x)\vert\sum_{y\in X }\sqrt{\frac{m(y)}{m(x)}}\mathrm{e}^{-\varepsilon d(x,y)}\leq C'\Vert f\Vert_1.
\end{align*}
The last inequality is due to Lemma \ref{LemmaVolGrowth}. This proves the statement for $p = 1.$ The case $p = 2$ is clear. We obtain the general statement for $p\in [1,\infty]$ via interpolation and duality arguments.
\end{proof}
 We are finally able to prove Theorem \ref{lpIndependence} following \cite{BauerHuaKeller}.
 \begin{proof}[Proof of Theorem \ref{lpIndependence}]
    
    First, we show the inclusion $\sigma( H_p)\subseteq \sigma( H_2).$ For $z\in\rho( H_2)$, let $g_z^{(2)}$ be the kernel of $( H_2-z)^{-2}.$ Then for any fixed $x,y\in X$ the map $\rho( H_2)\rightarrow\mathbb{C},z\mapsto g^{(2)}_z(x,y)$ is analytic. From Lemma \ref{lp2bounded} we conclude that $( H_2-z)^{-2}$ is analytic on $\rho( H_2)$ as a family of $\ell^p$-bounded operators. Moreover, $( H_p-z)^{-2}$ is analytic as a family of $\ell^p$-bounded operators on $\rho( H_p)$ (see \cite[Lemma 3.2]{HempelVoigt1986}). However, both resolvents are consistent for $z$ with negative real part. By unique continuation of analytic functions, we conclude that the connected set $\rho( H_2)$ is contained in the maximal domain of analyticity of $( H_p-z)^{-2},$ which is $\rho( H_p).$\\
    Now, let $1\leq p< 2$ and $q$ be the respective H\"older conjugate. From the inclusion $\sigma(H_p)\subseteq\sigma(H_2)\subseteq [0,\infty)$ we conclude that $\rho(H_p) = \rho(H_q).$ Moreover, it follows from the Laplace transform (which holds for $q= \infty$ in the weak-$\ast$-sense) that $(H_p-z)^{-1} = (H_q-z)^{-1}$ on $\ell^p\cap\ell^q$ for all $z$ with $\mathrm{Re}z<0.$ From \cite[Corollary~1.4]{HempelVoigt1987} it follows that $(H_p-z)^{-1} = (H_q-z)^{-1}$ on $\ell^p\cap\ell^q$ for all $z\in\rho(H_p).$ Thus, by interpolation, $(H_p-z)^{-1}$ extends to a bounded operator on $\ell^2.$ By \cite[Theorem~1.2]{HempelVoigt1987} the map $\rho(H_p)\to\mathcal{B}(\ell^2),z\mapsto(H_p-z)^{-1}$ is analytic. On the other hand, the map $z\mapsto (H_2-z)^{-1}$ is analytic on $\rho(H_2),$ which is also its maximal domain of analyticity. Now, observe that $(H_p-z)^{-1} = (H_2-z)^{-1}$ for all $z$ with $\mathrm{Re}z<0,$ which follows once again from the Laplace transform. Thus, we conclude from unique continuation of analytic functions that $\rho(H_p)\subseteq \rho(H_2)$ which finishes the proof.
\end{proof}

  \eat{ From the inclusion $\sigma( H_p)\subseteq\sigma( H_2)\subseteq[0,\infty)$ we conclude that $\rho( H_p)$ and $\rho( H_q)$ are connected for any $1\leq p,q\leq\infty$. Moreover, $( H_p-z)^{-1}$ and $( H_q-z)^{-1}$ are consistent for all $z,\mathrm{Re}z<0$ as a consequence of the Laplace transform. By \cite[Corollary 1.4]{HempelVoigt1987} they are consistent on $\rho( H_p)\cap\rho( H_q).$ By means of interpolation for $1\leq p<\infty$ and $q$ its dual exponent, we get that $( H_p-z)^{-1}$ extends consistently to an $\ell^2$-bounded operator. This gives us a family of $\ell^2$-bounded operators that is analytic on $\rho( H_p)$ and consistent with $( H_2-z)^{-1}$ for all $z$, $\mathrm{Re}z<0$ and, therefore, on the connected set $\rho( H_p)\cap\rho( H_2).$ 
    As before, since $\rho( H_p)$ is connected,  we conclude from unique continuation that $\rho( H_p)\subseteq\rho( H_2).$}

\section{Weighted Simplicial Complexes}\label{s:weighted}

In this section, we introduce the fundamental concepts and notation for weighted simplicial complexes and their associated Laplacians, following the framework developed in \cite{BK}.

\subsection{Laplacians on Weighted Simplicial Complexes}
 
Let $V$ be a countable set and $\hat \Sigma$ be a subset of the power set $\mathcal{P}(V)$ consisting of finite subsets of $V.$ We call $\hat \Sigma$ a \textit{simplicial complex} if $\sigma\in\hat \Sigma$ and $\tau\subseteq \sigma$ implies $\tau\in\hat \Sigma.$ The set of elements of cardinality $k+1$ (called $k$-\emph{simplices}) is denoted by $\hat \Sigma_k$. \\
Let $m:\hat{\Sigma}\rightarrow(0,\infty)$ be  a function. We define $ \Sigma = \hat \Sigma$ if $\sum_{v\in  \hat\Sigma_{0}}m(v)<\infty$ and $ \Sigma = \hat \Sigma\backslash\lbrace\varnothing\rbrace$ otherwise and call $( \Sigma,m)$ a \textit{weighted simplicial complex}. 
We assume throughout that the weight $m$ is \textit{locally summable}, i.e., for all $\tau\in \Sigma$
\[\gamma(\tau)= \sum_{\sigma\succ\tau}m(\sigma)<\infty,\]
where for $\tau,\sigma\in \Sigma$ with $\tau\subset\sigma$ and $|\sigma\backslash\tau| = 1$, we write $\tau\prec\sigma$ and say $\tau$ is a \textit{face} of $\sigma$ or $\sigma$ is a \textit{coface} of $\tau$. The \emph{dimension} of a simplex is $\dim(\tau)=|\tau|-1$ and the dimension of the complex is  $\dim(\Sigma)=\sup_{\tau\in \Sigma}\dim(\tau)$.\\
We  define the function space
\begin{align*}
\F &= \lbrace\omega\in C(\Sigma)\mid\sum_{\sigma\succ \tau}m(\sigma)|\omega(\sigma)|<\infty~\text{for all}~\tau\in \Sigma\rbrace.
\end{align*}
 A linear map $\delta:D(\delta)=C(\Sigma)\rightarrow C(\Sigma)$ is called a \textit{coboundary operator} if, for all $\tau,\sigma\in \Sigma$, we have $\delta 1_\tau(\sigma)\in\lbrace -1,0,1\rbrace$, $\delta 1_\tau(\sigma)\neq 0$ iff $\tau\prec\sigma$, and \[\delta\delta = 0.\]
On $D(\partial)=\F$, we define the \textit{boundary operator} $\partial:\F\rightarrow C(\Sigma)$ as
\[\partial\omega(\rho) = \frac{1}{m(\rho)}\sum_{\tau\succ \rho}m(\tau)\omega(\tau)\theta(\rho,\tau),\]
where $\theta(\tau,\sigma)=\delta1_\tau(\sigma)$. 
Here we define the coboundary operator on functions rather than alternating forms, as traditionally done in the literature. However, $\delta$ stands in a one-to-one correspondence to the respective coboundary operator acting on alternating forms on the oriented simplicial complex. This was shown in \cite[Subsection 2.2]{BK}. The operator $\delta$ acting on functions is for our purposes more convenient and allows for a direct interpretation of the Laplacian as a signed Schr\"odinger operator, \cite[Theorem 3.12]{BK}. The ``orientation'' is encoded in the sign $\theta(\tau,\sigma).$

We consider the Banach spaces $\ell^p=\ell^p( \Sigma,m)$, $1\leq p\leq \infty$, defined as above with norms $\|\cdot\|_p$, where again $\ell^2$ is a Hilbert space  with inner product $\langle\omega,\eta\rangle = \sum_{ \Sigma}m\omega\overline{\eta}$.

We define the quadratic forms
\begin{align*}
\Q^+(\omega) &= \|\delta\omega\|^{2}_2, & D(\Q^+) &= \lbrace \omega\in C(\Sigma)\mid \|\delta\omega\|^{2}_2<\infty\rbrace,\\
\Q^-(\omega) &= \|\partial\omega\|^{2}_2, & D(\Q^-) &= \lbrace \omega\in \F\mid\|\partial\omega\|^{2}_2<\infty\rbrace,\\
\Q^H(\omega) &= \|\delta\omega+\partial\omega\|^{2}, & D(\Q^H) &= \lbrace \omega\in \F\mid\|\delta\omega+\partial\omega\|^{2}_2<\infty\rbrace.
\end{align*}
By \cite[Section~3.1]{BK}, these quadratic forms are densely defined, non-negative and closed when restricted to the Hilbert space $\ell^2.$
For $\circ\in\{\pm,H\}$, the corresponding Dirichlet Laplacians are defined as the self-adjoint operators associated to the closed quadratic form
\begin{align*}
Q^\circ  = \Q^\circ\mid_{D(Q^\circ)}, \qquad D(Q^{\circ} ) = \overline{C_c(\Sigma)}^{\Vert\cdot\Vert_{\Q^{\circ}}},\qquad\Vert\cdot\Vert^2_{\mathcal{Q}^{\circ}} = \mathcal{Q}^{\circ}+\Vert\cdot\Vert^2_2.
\end{align*}
We denote by $\Delta^\circ=\Delta^\circ_2$ the positive self-adjoint operator arising from the closed quadratic form $Q^{\circ}$. 
In \cite[Theorem 3.4]{BK} it is shown that these operators are restrictions of the formal Laplacians
\begin{align*}
\Delta^+ \omega=  \partial\delta \omega, \qquad \Delta^-  \omega= \delta\partial \omega, \qquad \Delta^H  \omega= (\delta\partial+\partial\delta) \omega,
\end{align*}
for  $ \omega$ in the respective domains. 

A key perspective developed in \cite{BK} is the representation of these Laplacians as signed Schr\"{o}dinger operators, i.e.,
there exist explicit edge weights $b^\circ$, magnetic potentials $o^\circ$ taking values in $\{\pm1\}$, and potentials $c^\circ$ such that the Laplacians can be expressed as Schr\"{o}dinger operators 
 \[\Delta^\circ = \mathcal{H}^\circ\] on appropriate domains \cite[Theorem~3.19]{BK}. This allows us to leverage techniques from the theory of Schr\"{o}dinger operators developed above in our analysis of these Laplacians. The term $c^H/m,$ given by 
\[\frac{c^H(\tau)}{m(\tau)} =  \sum_{\rho\prec\tau}\frac{m(\tau)}{m(\rho)}+\sum_{\sigma\succ\tau}\frac{m(\sigma)}{m(\tau)}-\sum_{\rho\prec\tau}\sum_{\tau\neq\tau'\succ\rho}\Big\vert\frac{m(\tau')}{m(\rho)}-\sum_{\sigma\succ\tau,\tau'}\frac{m(\sigma)}{m(\tau)}\Big\vert,\]
is known in the literature as \textit{Forman curvature} \cite{F,JM}. Hence, we say that the weighted simplicial complex has \textit{form bounded curvature} if there exists $M,C>0$ such that for all $\omega\in C_c( \Sigma)$
\[\|\sqrt{c^H_-\slash m}\omega\|^2_2\leq M\mathcal{Q}^H(\omega)+C\Vert\omega\Vert^2_2.\]
We fix an intrinsic metric $d^\circ$ for $b^\circ$ over $( \Sigma,m)$ with finite jump size (see \cite[Section~4 and Section~6]{BK} for various possible choices) for the rest of this section and define the volume of balls as in Subsection~\ref{sectionVolGrowth}. In the next subsections, we discuss applications of the results developed above to the Laplacians on weighted simplicial complexes.

\subsection{Results for Weighted Simplicial Complexes}
The results from the previous sections for magnetic Schrödinger operators carry over to Laplacians on weighted simplicial complexes via their representation as signed Schr\"{o}dinger operators, \cite[Lemma 3.13]{BK}.

Denote the bottom of the spectrum of $\Delta^\circ_2$ by $\lambda^\circ_2$ and the semigroups generated by $-\Delta^\circ_2$ on $\ell^2( \Sigma,m)$ by $S^\circ_2,$  $\circ\in\{\pm,H\}.$ Furthermore, let $p^\circ$ be the kernel of $S^\circ_2.$
 
First, we have the following Davies-Gaffney-Grigoryan estimate, which answers a question by \cite[Remark 1.1]{HL}. Recall the function $\zeta$ from Section~\ref{s:magneticschrodinger} defined as $\zeta(r) =  r  \mathrm{arcsinh}\left({ r} \right)-\sqrt{1+ r^2}+1,$ $r\geq 0.$ The following theorem is a direct consequence of Theorem~\ref{dgg}.

\begin{theorem}[Davies-Gaffney-Grigoryan]\label{dggSC}
    For all $\omega,\eta\in\ell^2$, $A=\mathrm{supp}\ \omega$, $B= \mathrm{supp}\ \eta$, one has
    \[\vert\langle\mathrm{e}^{-t \Delta^\circ_2}\omega,\eta\rangle\vert\leq\mathrm{exp}\left(-\lambda_2^\circ t-\frac{t}{s^2}\zeta \left(\frac{s d^\circ(A,B)}{t}\right)\right)\Vert \omega\Vert_2\Vert \eta\Vert_2.\]
    In particular, the kernel $p_t^\circ$ of the semigroup $S_2^\circ(t) = \mathrm{e}^{-t \Delta^\circ_2}$ satisfies for all $\sigma,\tau\in  \Sigma$,
    \[\vert p_t^\circ(\sigma,\tau)\vert\leq (m(\sigma)m(\tau))^{-1\slash 2}\mathrm{exp}\left(-\lambda_2^\circ t-\frac{t}{s^2}\zeta \left(\frac{s d^\circ(\sigma,\tau)}{t}\right)\right).\]
\end{theorem}

Next, we come to extension of the semigroups  to $\ell^p$ spaces. Whenever $S_2^\circ$ extends to a strongly continuous semigroup $S_p^\circ$ on $\ell^p$ for $1\leq p< \infty$, we denote its generator on $\ell^p$ by $\Delta_p^\circ$ and $\Delta_\infty^\circ = (\Delta_1^\circ)^*$ on $\ell^\infty.$ \\
Before applying any of the above results, let us briefly discuss the case when $\Delta^\circ_2$ is bounded. In \cite[Theorem 4.1]{BK} it is shown that $\Delta^\circ_2$ is bounded if \[\sup_{\Sigma}(\dim +2)\gamma\slash m<\infty\] and that this already characterizes boundedness whenever $\Sigma$ is finite dimensional, i.e., $\dim(\Sigma)=\sup_\Sigma\dim<\infty.$ In that case $S^\circ_2$ is given by the absolutely convergent sum
\[S^\circ_2(t) = \sum_{k = 0}^\infty\frac{(-t)^k{\Delta^\circ_2}^k}{k!}.\]
Clearly, if $\Delta^\circ_2$ extends to a bounded operator on $\ell^p$, $1\leq p\leq\infty,$ then so does $S^\circ_2$, and we obtain an analytic semigroup $S^\circ_p$ on $\ell^p$ that acts consistently as $S^\circ_2.$ Therefore, we only need to show that $\Delta^\circ_2$ extends to a bounded operator on $\ell^p$ and since $\Delta^\circ_2$ acts as a composition of $\delta$ and $\partial$ (\cite[Theorem 3.4]{BK}), we further reduce the problem to mere boundedness of $\delta$ and $\partial$ on $\ell^p.$ It turns out that this holds automatically for all $1\leq p\leq\infty$, without further assumptions aside from boundedness on $\ell^2.$

\begin{proposition}[Boundedness]\label{boundedness} For $k\geq 0$, let $\delta_k,\partial_k$ be the restrictions of $\delta,\partial$ to $C(\Sigma_k).$
\begin{itemize}
    \item [(a)]  $\Vert\delta_k\Vert_{\infty,\infty}\leq k+2$ and   $\Vert\partial_k\Vert_{1,1}\leq k+1.$
    \item [(b)]If $C\coloneqq\sup_\Sigma\gamma\slash m<\infty,$ then \[\Vert\delta\Vert_{1,1}\leq C,\qquad\Vert\partial\Vert_{\infty,\infty}\leq C.\]
    In particular, $\partial_k,\delta_k$ are bounded on $\ell^p$ for all $1\leq p\leq\infty.$ Moreover, if $D\coloneqq\sup_{\Sigma}(\dim+2)\gamma\slash m<\infty,$ then \[\Vert \Delta^\pm_{p}\Vert_{p,p}\leq D\quad\text{and}\quad\Vert\Delta^H_{p}\Vert_{p,p}\leq 2D\] for all $1\leq p\leq \infty.$
    \item [(c)]If $\dim(\Sigma)<\infty,$ then the following are equivalent:
    \begin{itemize}
        \item [(i)]$\sup_\Sigma\gamma\slash m<\infty.$
        \item [(ii)]$\delta$ is bounded for some (all) $1\leq p<\infty.$
        \item [(iii)]$\partial$ is bounded for some (all) $1<p\leq\infty.$
    \end{itemize}
\end{itemize}
\end{proposition}
\begin{proof}
Let $\omega\in C(\Sigma_k)$ and observe that $\delta$ acts as \[\delta\omega(\sigma) = \sum_{\tau\prec\sigma}\theta(\tau,\sigma)\omega(\tau),~\sigma\in\Sigma.\] This sum has exactly $\dim(\sigma)+1 = k+2$ summands. Now, the operator bounds in (a) and (b) for $\delta$ and $p = 1,\infty$ follow from a simple application of the triangle inequality and for $1<p<\infty$ by the Riesz-Thorin interpolation theorem. Similarly, one obtains the respective operator bounds for $\partial.$ Since $\mathcal{L}^{-}=\delta\partial$, $\mathcal{L}^{+}=\partial\delta$ and $\mathcal{L}^{H}=(\delta+\partial)^2$ are compositions of $\delta$ and $\partial,$ 
they are bounded on $\ell^p$ as well. We obtain $\|\mathcal{L}^\pm\|_{p,p}\leq D$ and $\|\mathcal{L}^H\|_{p,p}\leq 2D$ by combining the bounds for $\delta$ and $\partial$. Hence, they generate  strongly continuous semigroups on $\ell^p$ via the exponential power series which coincides with $S_2^\circ$ on $\ell^2\cap\ell^p$. Thus, their generators $\Delta^\circ_p$ coincide with $\mathcal{L}^\circ$ on $\ell^p$ and are, therefore, bounded with the respective bound.

As for the equivalence in (c), we get ``(i) $\Rightarrow$ (ii)'' from (b). We also note, that by duality and Stokes' theorem \cite[Theorem 2.9]{BK}, $\delta$ is bounded on $\ell^p$ iff $\partial$ is bounded on $\ell^q,$ where $q$ is the H\"older dual of $p.$ This shows the equivalence between (ii) and (iii). 
 Finally, we assume (iii), i.e., $\delta$ is bounded on $\ell^p$ for some $1\leq p<\infty.$ Then there is $D>0$ such that, for every $\tau\in \Sigma$, we have 
 \[Dm(\tau) = D\Vert 1_\tau\Vert_p^p\geq \Vert\delta 1_\tau\Vert_p^p = \sum_{\sigma\succ\tau}m(\sigma) = \gamma(\tau).\]
 This finishes the proof.
\end{proof}

We say a  function
$[0,\infty)\rightarrow \ell^p$, $t\mapsto \omega_t$  solves 
\begin{equation} \tag{HE$^\circ$} \label{HE_p}
     -\Delta^\circ_p \omega_t (x)= \frac{d}{dt}\omega_t(x)     
\end{equation}
for an initial conditions $\omega_0\in\ell^p$  for $\circ\in \{\pm,H\}$, $p\in [1,\infty]$, if $t\mapsto \omega_t(x)$ is  continuous on $[0,\infty)$, differentiable on $(0,\infty)$ and $\omega_t\in D(\Delta_p^\circ)$ for $t>0$ and satisfies 
\eqref{HE_p}  for all $t>0$ and $x\in \Sigma$.
\begin{remark}
    In general it is not clear whether $\Delta^\circ_p$ acts as a composition of $\delta$ and $\partial$ similar to $\Delta^\circ.$ For locally finite simplicial complexes, we may deduce this from Theorem \ref{actionHp} and \cite[Theorem 3.12]{BK} since then $\Delta^\circ_p = \mathcal{H}^\circ = \mathcal{L}^\circ$ on $D(\Delta^\circ_p).$ Here, $\mathcal{L}^{-}=\delta\partial$, $\mathcal{L}^{+}=\partial\delta$ and $\mathcal{L}^{H}=(\delta+\partial)^2$ is the formal Laplacian as in the proof of  Proposition \ref{boundedness}. This applies in particular to combinatorial simplicial complexes, where $m\equiv 1.$
\end{remark}

The following theorem is a direct consequence of Theorem~\ref{formboundedpotential}, Theorem~\ref{boundedpotential}, Theorem~\ref{lpextend}, Proposition~\ref{analyticity} and Proposition~\ref{boundedness}.

\begin{theorem}[Semigroup extension]\label{semigroupextensioncomplexes} Let $\circ\in\{\pm,H\}$.
    \begin{itemize}
        \item[(a)] If the weighted simplicial complex has form bounded curvature with bound $M>0$, then for all  $p\in I =   [2\frac{M+1}{M}-2\frac{\sqrt{M+1}}{M},2\frac{M+1}{M}+2\frac{\sqrt{M+1}}{M}]$ the semigroup $S_2^\circ$ extends consistently to a strongly continuous semigroup $S_p^\circ$ on $\ell^p.$ In the case of bounded curvature, i.e., $M=0,$ this holds for all $p\in I = [1,\infty).$ In both cases $S_p^\circ(t)\omega_0$ solves \eqref{HE_p} for all  $\omega_0\in\ell^p$ and $p$ in the interior of $I.$
        \item[(b)] If the weighted simplicial complex has exponential volume growth, then the semigroup $S_2^\circ$ extends consistently to a strongly continuous semigroup $S_p^\circ$ on $\ell^p$ for all $1\leq p< \infty$ and $S_p^\circ(t)\omega_0$ solves \eqref{HE_p} for all  $\omega_0\in\ell^p$ and  all $1<p<\infty.$
        \item [(c)]If $\sup_\Sigma(\dim+2)\gamma\slash m<\infty,$ then the semigroup $S^\circ_2$ extends consistently to an analytic semigroup $S^\circ_p$ on $\ell^p$  and $S_p^\circ(t)\omega_0$ solves \eqref{HE_p} for  all  $\omega_0\in\ell^p$ and all $1\leq p\leq\infty.$
    \end{itemize}
\end{theorem}

Finally, we have the following result on the  $\ell^p$-spectrum. We start with the inclusion of the spectra under the assumption of uniformly positive measure which we directly deduce from Theorem~\ref{finitemeasure}.

\begin{theorem}\label{finitemeasureSC} Let  $ 1\leq p< \infty $. Assume that  $ S_{2} ^\circ$, $\circ\in\{\pm,H\}$, extends to a strongly continuous semigroup on $ \ell^{p} $  and  $\inf_{\tau\in \Sigma} m(\tau)>0.$ Then,
    \[\sigma( \Delta^\circ_2)\subseteq\sigma( \Delta^\circ_p).\]
\end{theorem}

\begin{remark}
  Again,  if $S_2^{\circ}$ extends to a strongly continuous semigroup on $\ell^1$, then by duality, we also have
    $\sigma( \Delta_2^{\circ})\subseteq  \sigma( \Delta^{\circ}_\infty) .$
\end{remark}

Finally, we come to  independence of the $\ell^p$-spectrum under form bounded curvature and subexponential volume growth. This follows directly from Theorem~\ref{lpIndependence}.

\begin{theorem}[Independence of $\ell^p$-Spectrum]\label{lpIndependenceSC}
   Let $\circ\in\{\pm,H\}.$ If $c^\circ\slash m$ is form bounded  and the graph $b^\circ$ has uniform subexponential volume growth with respect to $d^\circ,$ then for all $1\leq p\leq\infty$ \[\sigma( \Delta^\circ_2) = \sigma( \Delta^\circ_p).\]
\end{theorem}

\subsection{Results for Combinatorial Weights}\label{s:combinatorial}

Finally, we discuss the special case of combinatorial weights on simplicial complexes from the introduction. We say that a weighted simplicial complex $( \Sigma,m)$ has \textit{combinatorial weights} if $m\equiv1$. Note that this implies local finiteness of $\Sigma,$ i.e., every simplex has a finite number of cofaces, by the local summability assumption on $m.$ Below, we prove Theorem~\ref{thm:HE} and~\ref{thm:p-independence} from the introduction.
Recall that here we consider the combinatorial graph metric
\[d(v,w)=\min\{ n\mid \mbox{there is a path } {v=v_{0}\sim\ldots\sim v_{n}=w}\}\]
defined on the $1$-skeleton of $\Sigma,$ i.e., the collection of vertices and $1$-simplices, that we identify as a graph.
In order to apply the more abstract results about weighted simplicial complexes, we need to construct an intrinsic metric on all of $\Sigma$ just from $d.$ The remarkable fact here is that this is possible by considering the $1$-dimensional data of the complex only.

\begin{lemma}\label{LemmaSubexpVolGrowthSC}
      Let $\Sigma$ be a weighted simplicial complex with combinatorial weights. If the  volume of balls  grows at most (sub)exponentially with respect to the combinatorial metric $d$, then $\Sigma$ is finite dimensional and there is an intrinsic metric $\rho$ with finite jump size on $\Sigma$ such that the volume of balls grows at most (sub)exponentially with respect to $\rho$. 
\end{lemma}
\begin{proof}
By Lemma \ref{LemmaCombVolGrowth} the combinatorial vertex degree is uniformly bounded by some $D>0.$ This also implies that $\Sigma$ is finite dimensional. Therefore, the combinatorial metric divided by $\sqrt{D}$ is intrinsic with finite jump size, and we denote this metric by $\rho$. We extend $\rho $ to all of $\Sigma$ by
\[\rho(\sigma,\sigma') =(\dim(\Sigma)+1)^{-1/2}\inf_{\sigma=\sigma_{0}\sim\ldots\sim \sigma_{n}=\sigma'}\sum_{k=1}^{n}\max_{v\in\sigma_k\cap\sigma_{k-1} }\deg(v)^{-1/2}
 \]
for all $\sigma,\sigma'$ of the same dimension and by infinity for simplices of different dimension. In \cite[Proof of Theorem~6.1]{BK}, we have shown that $\rho$ is intrinsic on $\Sigma_k$, $1\le k\le \dim(\Sigma).$ Clearly, it has finite jump size as well. Furthermore, it was shown there that for all $\sigma,\sigma'\in\Sigma_k$, we find $v\in\sigma,v'\in\sigma'$ such that 
\[d(v,v')\slash\sqrt D = \rho(v,v')\leq  (\dim(\Sigma)+1)^{1/2}\rho(\sigma,\sigma').\]
Thus, we can conclude the (sub-)exponential volume growth of balls with respect to $\rho$ from that of $d$ since every $\sigma$ contains exactly $\dim(\sigma)+1$ vertices.
\end{proof}

\begin{lemma}\label{LemmaVolGrowthSC}
    Let $\Sigma$ be a weighted simplicial complex with combinatorial weights. Then, the following statements are equivalent:
    \begin{itemize}
        \item[(i)] The volume of balls grows at most exponentially with respect to some intrinsic metric with finite jump size on the $1$-skeleton.
        \item[(ii)]   The volume of balls grows at most exponentially with respect to the combinatorial metric on the $1$-skeleton.
        \item[(iii)] The combinatorial degree on the  $1$-skeleton is bounded.
        \item[(iv)] The coboundary operator $\delta$ is bounded on $\ell^p$ for some (all) $1\leq p< \infty.$
    \end{itemize}
    In that case there exists an intrinsic metric of finite jump size on $\Sigma$ such that the volume of balls grows at most exponentially.   Moreover, $\Delta^\circ$ is bounded on $\ell^p$ for all $1\leq p\leq\infty$ and $\circ\in\lbrace\pm,H\rbrace.$
\end{lemma}
\begin{proof}
(i) $\Rightarrow$ (iii): This follows from Lemma~\ref{LemmaCombVolGrowth}.

(iii) $\Rightarrow$ (ii):  If the combinatorial degree is bounded by $D\ge 1,$ then the combinatorial metric divided by $\sqrt{D}$ is intrinsic with finite jump size. Hence, the balls grow at most exponentially with respect to this metric, i.e., $\#B_r(x)\leq D^{r\sqrt{D}}$.

(ii) $\Rightarrow$ (i): In the proof of Lemma \ref{LemmaSubexpVolGrowthSC} we have seen that the combinatorial metric is intrinsic after normalization with some constant that depends on the dimension of $\Sigma$ only.

(iii) $\Rightarrow$ (iv):
Let the combinatorial degree be bounded by $D.$ This already forces $\Sigma$ to be finite dimensional as well as $\gamma\leq D.$ Thus, we get for every $\tau\in\Sigma$ the estimate \[\gamma(\tau)(\dim(\tau)+2)\leq D(\dim\Sigma+2)<\infty.\] We now conclude (iv) from Proposition~\ref{boundedness}.

(iv) $\Rightarrow$ (iii): Let $\delta$ be bounded on $\ell^p$ for some $1\leq p<\infty$ and choose an arbitrary vertex $v.$ Since the combinatorial degree of $v$ is precisely given by $\Vert \delta 1_v\Vert_p^p,$ we conclude (iii). 

The last part of the statement follows from Lemma \ref{LemmaSubexpVolGrowthSC} and Proposition \ref{boundedness}.
\end{proof}

We finish the paper with the proof of Theorem~\ref{thm:HE} and~\ref{thm:p-independence} from the introduction.

\begin{proof}[Proof of Theorem~\ref{thm:HE}]
 By Theorem~\ref{actionHp} and \cite[Theorem 3.12]{BK} (see also the remark after Proposition \ref{boundedness}) we have that the generator $\Delta_p^H$ acts as the formal Laplacian $\partial\delta+\delta\partial$ on $D(\Delta_p^H)$ for $1\leq p\leq \infty$ in case the $\ell^2$ semigroups extend to $\ell^p$. 
Assertion (a) follows directly from  Theorem~\ref{semigroupextensioncomplexes}~(a).
For assertion (b)  in the case of expontial volume growth, we get from the lemma above that the combinatorial degree on the $1$-skeleton is bounded by some $D>0.$ This implies that $\Sigma$ is finite dimensional and $\gamma\leq D.$ We conclude the statement now from Theorem \ref{semigroupextensioncomplexes} (c).
\end{proof}

\begin{proof}[Proof of Theorem~\ref{thm:p-independence}]
By Lemma~\ref{LemmaSubexpVolGrowthSC}, there exists an intrinsic metric on $\Sigma$ with finite jump size such that the volume of balls grows at most subexponentially. Furthermore, by  Lemma~\ref{LemmaVolGrowthSC}, we obtain that $\Delta^H$ is bounded on $\ell^1.$ Therefore,  $S_2^H$ extends to a strongly continuous semigroup $S_1^H$ on $\ell^1$ with $\Vert S_1^H(t)\Vert_{1,1}\leq \mathrm{e}^{t\Vert \Delta^H\Vert_{1,1}}.$ Hence, by Proposition~\ref{lowerboundpotential}, the curvature is bounded from below and, in particular, form bounded. Thus, the result is  a consequence of Theorem~\ref{lpIndependenceSC}.
\end{proof}

{\bf Acknowledgements:} The authors acknowledge support by the DFG and thank Delio Mugnolo for a hint on the literature. The second author acknowledges the financial support and hospitality of the IIAS, Jerusalem.

\bibliographystyle{alpha}
\bibliography{literature}

@book{deRham,
    author = {de Rham, Georges},
    title = {Differentiable Manifolds: Forms, Currents, Harmonic Forms},
    publisher = {Springer},
    year = {1984}
}

@article {Do83,
    AUTHOR = {Dodziuk, Jozef},
     TITLE = {Maximum principle for parabolic inequalities and the heat flow
              on open manifolds},
   JOURNAL = {Indiana Univ. Math. J.},
  FJOURNAL = {Indiana University Mathematics Journal},
    VOLUME = {32},
      YEAR = {1983},
    NUMBER = {5},
     PAGES = {703--716},
      ISSN = {0022-2518,1943-5258},
   MRCLASS = {58G11 (35K05)},
  MRNUMBER = {711862},
MRREVIEWER = {P.\ G\"unther},
       DOI = {10.1512/iumj.1983.32.32046},
       URL = {https://doi.org/10.1512/iumj.1983.32.32046},
}

@incollection {DM97,
    AUTHOR = {Dodziuk, Jozef and Mathai, Varghese},
     TITLE = {Approximating {$L^2$} invariants of amenable covering spaces:
              a heat kernel approach},
 BOOKTITLE = {Lipa's legacy ({N}ew {Y}ork, 1995)},
    SERIES = {Contemp. Math.},
    VOLUME = {211},
     PAGES = {151--167},
 PUBLISHER = {Amer. Math. Soc., Providence, RI},
      YEAR = {1997},
      ISBN = {0-8218-0671-8},
   MRCLASS = {58G11 (58G18 58G25)},
  MRNUMBER = {1476985},
MRREVIEWER = {Edward\ L.\ Bueler},
       DOI = {10.1090/conm/211/02818},
       URL = {https://doi.org/10.1090/conm/211/02818},
}

@article {DM98,
    AUTHOR = {Dodziuk, Jozef and Mathai, Varghese},
     TITLE = {Approximating {$L^2$} invariants of amenable covering spaces:
              a combinatorial approach},
   JOURNAL = {J. Funct. Anal.},
  FJOURNAL = {Journal of Functional Analysis},
    VOLUME = {154},
      YEAR = {1998},
    NUMBER = {2},
     PAGES = {359--378},
      ISSN = {0022-1236,1096-0783},
   MRCLASS = {58G26 (43A07 57M10)},
  MRNUMBER = {1612713},
MRREVIEWER = {Tadeusz\ Januszkiewicz},
       DOI = {10.1006/jfan.1997.3205},
       URL = {https://doi.org/10.1006/jfan.1997.3205},
}

@article{Eckmann1944,
    author = {Eckmann, Beno},
    title = {Harmonische Funktionen und Randwertaufgaben in einem Komplex},
    journal = {Comment. Math. Helv.},
    volume = {17},
    pages = {240--255},
    year = {1944}
}

@book {EngelNagelShort,
    AUTHOR = {Engel, Klaus-Jochen and Nagel, Rainer},
     TITLE = {A short course on operator semigroups},
    SERIES = {Universitext},
 PUBLISHER = {Springer, New York},
      YEAR = {2006},
     PAGES = {x+247},
      ISBN = {978-0387-31341-2; 0-387-31341-9},
   MRCLASS = {47-01 (47D03 47D06)},
  MRNUMBER = {2229872},
MRREVIEWER = {Jacek\ Banasiak},
}

@article{HempelVoigt1986,
    author = {Hempel, R. and Voigt, J.},
    title = {The Spectrum of a {S}chr\"odinger Operator in ${L}^p(\mathbb{R}^d)$ is $p$-Independent},
    journal = {Commun. Math. Phys.},
    volume = {104},
    pages = {243--250},
    year = {1986}
}

@article {HempelVoigt1987,
    AUTHOR = {Hempel, Rainer and Voigt, J\"urgen},
     TITLE = {On the {$L_p$}-spectrum of {S}chr\"odinger operators},
   JOURNAL = {J. Math. Anal. Appl.},
  FJOURNAL = {Journal of Mathematical Analysis and Applications},
    VOLUME = {121},
      YEAR = {1987},
    NUMBER = {1},
     PAGES = {138--159},
      ISSN = {0022-247X},
   MRCLASS = {35P05 (47F05 81C35)},
  MRNUMBER = {869525},
MRREVIEWER = {Shu\ Nakamura},
       DOI = {10.1016/0022-247X(87)90244-7},
       URL = {https://doi.org/10.1016/0022-247X(87)90244-7},
}

@book{jost2026spectra,
  author = {Jost, J\"urgen and Mulas, Raffaella and Zhang, Dong},
  title = {Spectra of Discrete Structures},
  series = {Cambridge Studies in Advanced Mathematics},
  publisher = {Cambridge University Press},
  address = {Cambridge},
  year = {2026},
  note = {Forthcoming},
}

@incollection {MHJ,
    AUTHOR = {Mulas, Raffaella and Horak, Danijela and Jost, J\"urgen},
     TITLE = {Graphs, simplicial complexes and hypergraphs: spectral theory
              and topology},
 BOOKTITLE = {Higher-order systems},
    SERIES = {Underst. Complex Syst.},
     PAGES = {1--58},
 PUBLISHER = {Springer, Cham},
      YEAR = {[2022] \copyright 2022},
      ISBN = {978-3-030-91373-1; 978-3-030-91374-8},
   MRCLASS = {05-01 (05C10 05C50 05C90 05E45)},
  MRNUMBER = {4433789},
       DOI = {10.1007/978-3-030-91374-8\_1},
       URL = {https://doi.org/10.1007/978-3-030-91374-8_1},
}

@article {MKBJ,
    AUTHOR = {Mulas, Raffaella and Kuehn, Christian and B\"ohle, Tobias and
              Jost, J\"urgen},
     TITLE = {Random walks and {L}aplacians on hypergraphs: when do they
              match?},
   JOURNAL = {Discrete Appl. Math.},
  FJOURNAL = {Discrete Applied Mathematics. The Journal of Combinatorial
              Algorithms, Informatics and Computational Sciences},
    VOLUME = {317},
      YEAR = {2022},
     PAGES = {26--41},
      ISSN = {0166-218X,1872-6771},
   MRCLASS = {05C50 (05C65 05C81)},
  MRNUMBER = {4420488},
MRREVIEWER = {Anirban\ Banerjee},
       DOI = {10.1016/j.dam.2022.04.009},
       URL = {https://doi.org/10.1016/j.dam.2022.04.009},
}

@article {ATH,
    AUTHOR = {Ann\'e, Colette and Torki-Hamza, Nabila},
     TITLE = {The {G}auss-{B}onnet operator of an infinite graph},
   JOURNAL = {Anal. Math. Phys.},
  FJOURNAL = {Analysis and Mathematical Physics},
    VOLUME = {5},
      YEAR = {2015},
    NUMBER = {2},
     PAGES = {137--159},
      ISSN = {1664-2368,1664-235X},
   MRCLASS = {39A70 (05C12 05C50 05C63 47B25)},
  MRNUMBER = {3344097},
       DOI = {10.1007/s13324-014-0090-0},
       URL = {https://doi.org/10.1007/s13324-014-0090-0},
}

@article{BK,
  title={{O}n {H}odge {L}aplacians on {G}eneral {S}implicial {C}omplexes},
  author={Bartmann, Philipp and Keller, Matthias},
  journal={arXiv preprint arXiv:2508.07761},
  year={2025}
}

@book {Barlow,
    AUTHOR = {Barlow, Martin T.},
     TITLE = {Random walks and heat kernels on graphs},
    SERIES = {London Mathematical Society Lecture Note Series},
    VOLUME = {438},
 PUBLISHER = {Cambridge University Press, Cambridge},
      YEAR = {2017},
     PAGES = {xi+226},
      ISBN = {978-1-107-67442-4},
   MRCLASS = {60-02 (05C63 05C81 60J10 60J45)},
  MRNUMBER = {3616731},
MRREVIEWER = {Nicolas\ Curien},
       DOI = {10.1017/9781107415690},
       URL = {https://doi.org/10.1017/9781107415690},
}

@article {BauerHuaKeller,
    AUTHOR = {Bauer, Frank and Hua, Bobo and Keller, Matthias},
     TITLE = {On the {$l^p$} spectrum of {L}aplacians on graphs},
   JOURNAL = {Adv. Math.},
  FJOURNAL = {Advances in Mathematics},
    VOLUME = {248},
      YEAR = {2013},
     PAGES = {717--735},
      ISSN = {0001-8708,1090-2082},
   MRCLASS = {47A10 (05C50)},
  MRNUMBER = {3107525},
MRREVIEWER = {Yufei\ Huang},
       DOI = {10.1016/j.aim.2013.05.029},
       URL = {https://doi.org/10.1016/j.aim.2013.05.029},
}

@article {BHY1,
    AUTHOR = {Bauer, Frank and Hua, Bobo and Yau, Shing-Tung},
     TITLE = {Davies-{G}affney-{G}rigor'yan lemma on graphs},
   JOURNAL = {Comm. Anal. Geom.},
  FJOURNAL = {Communications in Analysis and Geometry},
    VOLUME = {23},
      YEAR = {2015},
    NUMBER = {5},
     PAGES = {1031--1068},
      ISSN = {1019-8385,1944-9992},
   MRCLASS = {58J35},
  MRNUMBER = {3458812},
MRREVIEWER = {Qihua\ Ruan},
       DOI = {10.4310/CAG.2015.v23.n5.a4},
       URL = {https://doi.org/10.4310/CAG.2015.v23.n5.a4},
}

@article {BHY2,
    AUTHOR = {Bauer, Frank and Hua, Bobo and Yau, Shing-Tung},
     TITLE = {Sharp {D}avies-{G}affney-{G}rigor'yan lemma on graphs},
   JOURNAL = {Math. Ann.},
  FJOURNAL = {Mathematische Annalen},
    VOLUME = {368},
      YEAR = {2017},
    NUMBER = {3-4},
     PAGES = {1429--1437},
      ISSN = {0025-5831,1432-1807},
   MRCLASS = {58J35 (05C63)},
  MRNUMBER = {3673659},
MRREVIEWER = {Thierry\ Coulhon},
       DOI = {10.1007/s00208-017-1529-z},
       URL = {https://doi.org/10.1007/s00208-017-1529-z},
}

@article {Braun,
    AUTHOR = {Braun, Mathias},
     TITLE = {Heat flow on 1-forms under lower {R}icci bounds. {F}unctional
              inequalities, spectral theory, and heat kernel},
   JOURNAL = {J. Funct. Anal.},
  FJOURNAL = {Journal of Functional Analysis},
    VOLUME = {283},
      YEAR = {2022},
    NUMBER = {7},
     PAGES = {Paper No. 109599, 65},
      ISSN = {0022-1236,1096-0783},
   MRCLASS = {35K08 (35P15 58A14 58C40 58J35)},
  MRNUMBER = {4444737},
       DOI = {10.1016/j.jfa.2022.109599},
       URL = {https://doi.org/10.1016/j.jfa.2022.109599},
}

@article {CHQSZ,
    AUTHOR = {Cao, Shiping and Huang, Yiqi and Qiu, Hua and Strichartz,
              Robert S. and Zhu, Xiaohan},
     TITLE = {Spectral analysis beyond {$\ell^2$} on {S}ierpinski lattices},
   JOURNAL = {J. Fourier Anal. Appl.},
  FJOURNAL = {The Journal of Fourier Analysis and Applications},
    VOLUME = {27},
      YEAR = {2021},
    NUMBER = {3},
     PAGES = {Paper No. 55, 19},
      ISSN = {1069-5869,1531-5851},
   MRCLASS = {47A10 (28A80 47B37)},
  MRNUMBER = {4269450},
MRREVIEWER = {Saurabh\ Verma},
       DOI = {10.1007/s00041-021-09853-y},
       URL = {https://doi.org/10.1007/s00041-021-09853-y},
}

@article {Ch,
    AUTHOR = {Charalambous, Nelia},
     TITLE = {On the {$L^p$} independence of the spectrum of the {H}odge
              {L}aplacian on non-compact manifolds},
   JOURNAL = {J. Funct. Anal.},
  FJOURNAL = {Journal of Functional Analysis},
    VOLUME = {224},
      YEAR = {2005},
    NUMBER = {1},
     PAGES = {22--48},
      ISSN = {0022-1236,1096-0783},
   MRCLASS = {58J50 (58J35)},
  MRNUMBER = {2139103},
MRREVIEWER = {Chadwick\ Sprouse},
       DOI = {10.1016/j.jfa.2004.11.003},
       URL = {https://doi.org/10.1016/j.jfa.2004.11.003},
}

@article {ChL,
    AUTHOR = {Charalambous, Nelia and Lu, Zhiqin},
     TITLE = {{$L^p$}-spectral theory for the {L}aplacian on forms},
   JOURNAL = {J. Funct. Anal.},
  FJOURNAL = {Journal of Functional Analysis},
    VOLUME = {289},
      YEAR = {2025},
    NUMBER = {6},
     PAGES = {Paper No. 110976, 44},
      ISSN = {0022-1236,1096-0783},
   MRCLASS = {58J50 (47A10 58A10)},
  MRNUMBER = {4889202},
       DOI = {10.1016/j.jfa.2025.110976},
       URL = {https://doi.org/10.1016/j.jfa.2025.110976},
}

@article {C,
    AUTHOR = {Chebbi, Yassin},
     TITLE = {The discrete {L}aplacian of a 2-simplicial complex},
   JOURNAL = {Potential Anal.},
  FJOURNAL = {Potential Analysis. An International Journal Devoted to the
              Interactions between Potential Theory, Probability Theory,
              Geometry and Functional Analysis},
    VOLUME = {49},
      YEAR = {2018},
    NUMBER = {2},
     PAGES = {331--358},
      ISSN = {0926-2601,1572-929X},
   MRCLASS = {05C50 (05C10 31C20 39A12 47B25)},
  MRNUMBER = {3824965},
MRREVIEWER = {Aingeru\ Fern\'andez\ Bertolin},
       DOI = {10.1007/s11118-017-9659-1},
       URL = {https://doi.org/10.1007/s11118-017-9659-1},
}

@article {CC,
    AUTHOR = {Carron, Gilles and Coulhon, Thierry and Hassell, Andrew},
     TITLE = {Riesz transform and {$L^p$}-cohomology for manifolds with
              {E}uclidean ends},
   JOURNAL = {Duke Math. J.},
  FJOURNAL = {Duke Mathematical Journal},
    VOLUME = {133},
      YEAR = {2006},
    NUMBER = {1},
     PAGES = {59--93},
      ISSN = {0012-7094,1547-7398},
   MRCLASS = {58J50 (42B20 58J35)},
  MRNUMBER = {2219270},
MRREVIEWER = {Emmanuel\ Pedon},
       DOI = {10.1215/S0012-7094-06-13313-6},
       URL = {https://doi.org/10.1215/S0012-7094-06-13313-6},
}

@article {CKK,
    AUTHOR = {Chen, Zhen-Qing and Kim, Daehong and Kuwae, Kazuhiro},
     TITLE = {{$L^p$}-independence of spectral radius for generalized
              {F}eynman-{K}ac semigroups},
   JOURNAL = {Math. Ann.},
  FJOURNAL = {Mathematische Annalen},
    VOLUME = {374},
      YEAR = {2019},
    NUMBER = {1-2},
     PAGES = {601--652},
      ISSN = {0025-5831,1432-1807},
   MRCLASS = {60J25 (31C25 35J10 60J35 60J45 60J57)},
  MRNUMBER = {3961322},
MRREVIEWER = {Ren\'e\ L.\ Schilling},
       DOI = {10.1007/s00208-018-1746-0},
       URL = {https://doi.org/10.1007/s00208-018-1746-0},
}

@article {CCH,
    AUTHOR = {Chen, Li and Coulhon, Thierry and Hua, Bobo},
     TITLE = {Riesz transforms for bounded {L}aplacians on graphs},
   JOURNAL = {Math. Z.},
  FJOURNAL = {Mathematische Zeitschrift},
    VOLUME = {294},
      YEAR = {2020},
    NUMBER = {1-2},
     PAGES = {397--417},
      ISSN = {0025-5874,1432-1823},
   MRCLASS = {42B25 (47B38)},
  MRNUMBER = {4050071},
MRREVIEWER = {Dachun\ Yang},
       DOI = {10.1007/s00209-019-02253-5},
       URL = {https://doi.org/10.1007/s00209-019-02253-5},
}

@incollection {Davies,
    AUTHOR = {Davies, E. B.},
     TITLE = {Heat kernel bounds, conservation of probability and the
              {F}eller property},
      NOTE = {Festschrift on the occasion of the 70th birthday of Shmuel
              Agmon},
   JOURNAL = {J. Anal. Math.},
  FJOURNAL = {Journal d'Analyse Math\'ematique},
    VOLUME = {58},
      YEAR = {1992},
     PAGES = {99--119},
      ISSN = {0021-7670,1565-8538},
   MRCLASS = {58G11 (47D07 47F05 58G32)},
  MRNUMBER = {1226938},
MRREVIEWER = {Kazuaki\ Taira},
       DOI = {10.1007/BF02790359},
       URL = {https://doi.org/10.1007/BF02790359},

}

@book {DaviesLinOp,
    AUTHOR = {Davies, E. Brian},
     TITLE = {Linear operators and their spectra},
    SERIES = {Cambridge Studies in Advanced Mathematics},
    VOLUME = {106},
 PUBLISHER = {Cambridge University Press, Cambridge},
      YEAR = {2007},
     PAGES = {xii+451},
      ISBN = {978-0-521-86629-3; 0-521-86629-4},
   MRCLASS = {47-01 (47A10 47D07 47N30 60J10 60J27)},
  MRNUMBER = {2359869},
MRREVIEWER = {Florian\ Horia\ Vasilescu},
       DOI = {10.1017/CBO9780511618864},
       URL = {https://doi.org/10.1017/CBO9780511618864},
}

@article{ennaceur_jadlaoui_hodge,
  title = {Hodge Laplacians on Weighted Simplicial Complexes: Forms, Closures, and Essential Self-Adjointness},
  author = {Ennaceur, Marwa and Jadlaoui, Amel},
  journal = {arXiv preprint},
  year = {2025},
  eprint = {2510.15546},
  archiveprefix = {arXiv},
  primaryclass = {math.SP},
  url = {https://arxiv.org/abs/2510.15546}
}

@article{ennaceur_jadlaoui_geometric,
  title = {Geometric Criteria for Essential Self-Adjointness of Discrete Hodge Laplacians on Weighted Simplicial Complexes},
  author = {Ennaceur, Marwa and Jadlaoui, Amel},
  journal = {arXiv preprint},
  year = {2025},
  eprint = {2510.18661},
  archiveprefix = {arXiv},
  primaryclass = {math.SP},
  url = {https://arxiv.org/abs/2510.18661}
}

@article {F,
    AUTHOR = {Forman, Robin},
     TITLE = {Bochner's method for cell complexes and combinatorial {R}icci
              curvature},
   JOURNAL = {Discrete Comput. Geom.},
  FJOURNAL = {Discrete \& Computational Geometry. An International Journal
              of Mathematics and Computer Science},
    VOLUME = {29},
      YEAR = {2003},
    NUMBER = {3},
     PAGES = {323--374},
      ISSN = {0179-5376,1432-0444},
   MRCLASS = {52B70 (53C20)},
  MRNUMBER = {1961004},
       DOI = {10.1007/s00454-002-0743-x},
       URL = {https://doi.org/10.1007/s00454-002-0743-x},
}

@article {Gaffney,
    AUTHOR = {Gaffney, Matthew P.},
     TITLE = {The conservation property of the heat equation on {R}iemannian
              manifolds},
   JOURNAL = {Comm. Pure Appl. Math.},
  FJOURNAL = {Communications on Pure and Applied Mathematics},
    VOLUME = {12},
      YEAR = {1959},
     PAGES = {1--11},
      ISSN = {0010-3640,1097-0312},
   MRCLASS = {53.00 (35.00)},
  MRNUMBER = {102097},
       DOI = {10.1002/cpa.3160120102},
       URL = {https://doi.org/10.1002/cpa.3160120102},
}

@article {G,
    AUTHOR = {Gaffney, Matthew P.},
     TITLE = {The harmonic operator for exterior differential forms},
   JOURNAL = {Proc. Nat. Acad. Sci. U.S.A.},
  FJOURNAL = {Proceedings of the National Academy of Sciences of the United
              States of America},
    VOLUME = {37},
      YEAR = {1951},
     PAGES = {48--50},
      ISSN = {0027-8424},
   MRCLASS = {53.0X},
  MRNUMBER = {48138},
MRREVIEWER = {K.\ Kodaira},
       DOI = {10.1073/pnas.37.1.48},
       URL = {https://doi.org/10.1073/pnas.37.1.48},
}

@article {GKS,
    AUTHOR = {G\"uneysu, Batu and Keller, Matthias and Schmidt, Marcel},
     TITLE = {A {F}eynman-{K}ac-{I}t\^o{} formula for magnetic
              {S}chr\"odinger operators on graphs},
   JOURNAL = {Probab. Theory Related Fields},
  FJOURNAL = {Probability Theory and Related Fields},
    VOLUME = {165},
      YEAR = {2016},
    NUMBER = {1-2},
     PAGES = {365--399},
      ISSN = {0178-8051,1432-2064},
   MRCLASS = {39A70 (31C20 35J10 47D08 60H05 60J75 81Q35)},
  MRNUMBER = {3500274},
       DOI = {10.1007/s00440-015-0633-9},
       URL = {https://doi.org/10.1007/s00440-015-0633-9},
}

@book {grigoryanbook,
    AUTHOR = {Grigor'yan, Alexander},
     TITLE = {Introduction to analysis on graphs},
    SERIES = {University Lecture Series},
    VOLUME = {71},
 PUBLISHER = {American Mathematical Society, Providence, RI},
      YEAR = {2018},
     PAGES = {viii+150},
      ISBN = {978-1-4704-4397-9},
   MRCLASS = {60J10 (05C25 05C50 05C63 05C81 58J35)},
  MRNUMBER = {3822363},
MRREVIEWER = {Serguei\ Popov},
       DOI = {10.1090/ulect/071},
       URL = {https://doi.org/10.1090/ulect/071},
}

@article {GrigPath,
    AUTHOR = {Grigor'yan, Alexander and Lin, Yong and Yau, Shing-Tung and
              Zhang, Haohang},
     TITLE = {Eigenvalues of the {H}odge {L}aplacian on digraphs},
   JOURNAL = {Comm. Anal. Geom.},
  FJOURNAL = {Communications in Analysis and Geometry},
    VOLUME = {33},
      YEAR = {2025},
    NUMBER = {4},
     PAGES = {981--1023},
      ISSN = {1019-8385,1944-9992},
   MRCLASS = {05C50 (58J50)},
  MRNUMBER = {4951371},
       DOI = {10.4310/cag.250815152654},
       URL = {https://doi.org/10.4310/cag.250815152654},
}

@article {GHM,
    AUTHOR = {Grigor'yan, Alexander and Huang, Xueping and Masamune, Jun},
     TITLE = {On stochastic completeness of jump processes},
   JOURNAL = {Math. Z.},
  FJOURNAL = {Mathematische Zeitschrift},
    VOLUME = {271},
      YEAR = {2012},
    NUMBER = {3-4},
     PAGES = {1211--1239},
      ISSN = {0025-5874,1432-1823},
   MRCLASS = {60J75 (60J25 60J27 60J45)},
  MRNUMBER = {2945605},
MRREVIEWER = {Longmin\ Wang},
       DOI = {10.1007/s00209-011-0911-x},
       URL = {https://doi.org/10.1007/s00209-011-0911-x},
}

@article {HJ,
    AUTHOR = {Horak, Danijela and Jost, J\"urgen},
     TITLE = {Spectra of combinatorial {L}aplace operators on simplicial
              complexes},
   JOURNAL = {Adv. Math.},
  FJOURNAL = {Advances in Mathematics},
    VOLUME = {244},
      YEAR = {2013},
     PAGES = {303--336},
      ISSN = {0001-8708,1090-2082},
   MRCLASS = {55U10 (18G30 31C20)},
  MRNUMBER = {3077874},
MRREVIEWER = {Paul\ G.\ Goerss},
       DOI = {10.1016/j.aim.2013.05.007},
       URL = {https://doi.org/10.1016/j.aim.2013.05.007},
}

@article {HL,
    AUTHOR = {Hua, Bobo and Luo, Xin},
     TITLE = {Davies-{G}affney-{G}rigor'yan lemma on simplicial complexes},
   JOURNAL = {Math. Z.},
  FJOURNAL = {Mathematische Zeitschrift},
    VOLUME = {290},
      YEAR = {2018},
    NUMBER = {3-4},
     PAGES = {1041--1053},
      ISSN = {0025-5874,1432-1823},
   MRCLASS = {58J35 (05E45 35J05 35K08 55U10)},
  MRNUMBER = {3856843},
MRREVIEWER = {J\'ozef\ Dodziuk},
       DOI = {10.1007/s00209-018-2051-z},
       URL = {https://doi.org/10.1007/s00209-018-2051-z},
}

@article{JM,
  title={Characterizations of Forman curvature},
  author={Jost, J{\"u}rgen and M{\"u}nch, Florentin},
  journal={arXiv preprint arXiv:2110.04554},
  year={2021}
}

@book {K,
    AUTHOR = {Kato, Tosio},
     TITLE = {Perturbation theory for linear operators},
    SERIES = {Classics in Mathematics},
      NOTE = {Reprint of the 1980 edition},
 PUBLISHER = {Springer-Verlag, Berlin},
      YEAR = {1995},
     PAGES = {xxii+619},
      ISBN = {3-540-58661-X},
   MRCLASS = {47A55 (46-00 47-00)},
  MRNUMBER = {1335452},
}

@book {KLW,
    AUTHOR = {Keller, Matthias and Lenz, Daniel and Wojciechowski, Rados\l
              aw K.},
     TITLE = {Graphs and discrete {D}irichlet spaces},
    SERIES = {Grundlehren der mathematischen Wissenschaften},
    VOLUME = {358},
 PUBLISHER = {Springer, Cham},
      YEAR = {2021},
     PAGES = {xv+668},
      ISBN = {978-3-030-81458-8; 978-3-030-81459-5},
   MRCLASS = {05-01 (05C63 35P05)},
  MRNUMBER = {4383783},
       DOI = {10.1007/978-3-030-81459-5},
       URL = {https://doi.org/10.1007/978-3-030-81459-5},
}

@book {Lunardi,
    AUTHOR = {Lunardi, Alessandra},
     TITLE = {Interpolation theory},
    SERIES = {Appunti. Scuola Normale Superiore di Pisa (Nuova Serie)
              [Lecture Notes. Scuola Normale Superiore di Pisa (New
              Series)]},
    VOLUME = {16},
   EDITION = {Third},
 PUBLISHER = {Edizioni della Normale, Pisa},
      YEAR = {2018},
     PAGES = {xiv+199},
      ISBN = {978-88-7642-639-1; 978-88-7642-638-4},
   MRCLASS = {46M35 (46-02 46B70 47D06 47F05)},
  MRNUMBER = {3753604},
       DOI = {10.1007/978-88-7642-638-4},
       URL = {https://doi.org/10.1007/978-88-7642-638-4},
}

@article {Magniez,
    AUTHOR = {Magniez, Jocelyn},
     TITLE = {Riesz transforms of the {H}odge--de {R}ham {L}aplacian on
              {R}iemannian manifolds},
   JOURNAL = {Math. Nachr.},
  FJOURNAL = {Mathematische Nachrichten},
    VOLUME = {289},
      YEAR = {2016},
    NUMBER = {8-9},
     PAGES = {1021--1043},
      ISSN = {0025-584X,1522-2616},
   MRCLASS = {42B20 (31C12 47F05 58J35)},
  MRNUMBER = {3512047},
MRREVIEWER = {Elena\ Cordero},
       DOI = {10.1002/mana.201400307},
       URL = {https://doi.org/10.1002/mana.201400307},
}

@incollection {Masamune,
    AUTHOR = {Masamune, Jun},
     TITLE = {A {L}iouville property and its application to the {L}aplacian
              of an infinite graph},
 BOOKTITLE = {Spectral analysis in geometry and number theory},
    SERIES = {Contemp. Math.},
    VOLUME = {484},
     PAGES = {103--115},
 PUBLISHER = {Amer. Math. Soc., Providence, RI},
      YEAR = {2009},
      ISBN = {978-0-8218-4269-0},
   MRCLASS = {05C50 (05C63 39A70 58J05)},
  MRNUMBER = {1500141},
MRREVIEWER = {J\'ozef\ Dodziuk},
       DOI = {10.1090/conm/484/09468},
       URL = {https://doi.org/10.1090/conm/484/09468},
}

@article{MugnoloHyper,
  title={Non-Markovian heat flows on directed hypergraphs},
  author={Mugnolo, Delio},
  journal={arXiv preprint arXiv:2510.17497},
  year={2025}
}

@incollection {O,
    AUTHOR = {Ollivier, Yann},
     TITLE = {A survey of {R}icci curvature for metric spaces and {M}arkov
              chains},
 BOOKTITLE = {Probabilistic approach to geometry},
    SERIES = {Adv. Stud. Pure Math.},
    VOLUME = {57},
     PAGES = {343--381},
 PUBLISHER = {Math. Soc. Japan, Tokyo},
      YEAR = {2010},
      ISBN = {978-4-931469-58-7},
   MRCLASS = {58J65 (46E35 53C23 60B05 60J10)},
  MRNUMBER = {2648269},
MRREVIEWER = {Mu\ Fa\ Chen},
       DOI = {10.2969/aspm/05710343},
       URL = {https://doi.org/10.2969/aspm/05710343},
}

@incollection {S,
    AUTHOR = {Schmidt, Marcel},
     TITLE = {On the existence and uniqueness of self-adjoint realizations
              of discrete (magnetic) {S}chr\"odinger operators},
 BOOKTITLE = {Analysis and geometry on graphs and manifolds},
    SERIES = {London Math. Soc. Lecture Note Ser.},
    VOLUME = {461},
     PAGES = {250--327},
 PUBLISHER = {Cambridge Univ. Press, Cambridge},
      YEAR = {2020},
      ISBN = {978-1-108-71318-4},
   MRCLASS = {05C50},
  MRNUMBER = {4412977},
}

@article {Strichartz,
    AUTHOR = {Strichartz, Robert S.},
     TITLE = {Analysis of the {L}aplacian on the complete {R}iemannian
              manifold},
   JOURNAL = {J. Functional Analysis},
  FJOURNAL = {Journal of Functional Analysis},
    VOLUME = {52},
      YEAR = {1983},
    NUMBER = {1},
     PAGES = {48--79},
      ISSN = {0022-1236},
   MRCLASS = {58G11 (47D05 58G25)},
  MRNUMBER = {705991},
MRREVIEWER = {J\'ozef\ Dodziuk},
       DOI = {10.1016/0022-1236(83)90090-3},
       URL = {https://doi.org/10.1016/0022-1236(83)90090-3},
}

@book {T,
    AUTHOR = {Thaller, Bernd},
     TITLE = {The {D}irac equation},
    SERIES = {Texts and Monographs in Physics},
 PUBLISHER = {Springer-Verlag, Berlin},
      YEAR = {1992},
     PAGES = {xviii+357},
      ISBN = {3-540-54883-1},
   MRCLASS = {81Q10 (35Q40 47N50 58G25)},
  MRNUMBER = {1219537},
MRREVIEWER = {P.\ A.\ Mishnayevskiy},
       DOI = {10.1007/978-3-662-02753-0},
       URL = {https://doi.org/10.1007/978-3-662-02753-0},
}

@book {Stein,
    AUTHOR = {Stein, Elias M.},
     TITLE = {Topics in harmonic analysis related to the
              {L}ittlewood-{P}aley theory},
    SERIES = {Annals of Mathematics Studies},
    VOLUME = {No. 63},
 PUBLISHER = {Princeton University Press, Princeton, NJ; University of Tokyo
              Press, Tokyo},
      YEAR = {1970},
     PAGES = {viii+146},
   MRCLASS = {42.50 (22.00)},
  MRNUMBER = {252961},
MRREVIEWER = {R.\ E.\ Edwards},
}

@article {Sturm,
    AUTHOR = {Sturm, Karl-Theodor},
     TITLE = {On the {$L^p$}-spectrum of uniformly elliptic operators on
              {R}iemannian manifolds},
   JOURNAL = {J. Funct. Anal.},
  FJOURNAL = {Journal of Functional Analysis},
    VOLUME = {118},
      YEAR = {1993},
    NUMBER = {2},
     PAGES = {442--453},
      ISSN = {0022-1236,1096-0783},
   MRCLASS = {58G25 (35P05 58G03)},
  MRNUMBER = {1250269},
MRREVIEWER = {Alberto\ G.\ Setti},
       DOI = {10.1006/jfan.1993.1150},
       URL = {https://doi.org/10.1006/jfan.1993.1150},
}

@book {Wojc,
    AUTHOR = {Wojciechowski, Radoslaw Krzysztof},
     TITLE = {Stochastic completeness of graphs},
      NOTE = {Thesis (Ph.D.)--City University of New York},
 PUBLISHER = {ProQuest LLC, Ann Arbor, MI},
      YEAR = {2008},
     PAGES = {87},
      ISBN = {978-0549-58579-4},
   MRCLASS = {99-05},
  MRNUMBER = {2711706},
       URL =
              {http://gateway.proquest.com/openurl?url_ver=Z39.88-2004&rft_val_fmt=info:ofi/fmt:kev:mtx:dissertation&res_dat=xri:pqdiss&rft_dat=xri:pqdiss:3310649},
}

@article {Dodziuk,
    AUTHOR = {Dodziuk, Jozef},
     TITLE = {Finite-difference approach to the {H}odge theory of harmonic
              forms},
   JOURNAL = {Amer. J. Math.},
  FJOURNAL = {American Journal of Mathematics},
    VOLUME = {98},
      YEAR = {1976},
    NUMBER = {1},
     PAGES = {79--104},
      ISSN = {0002-9327,1080-6377},
   MRCLASS = {58A10 (58G99)},
  MRNUMBER = {407872},
MRREVIEWER = {D.\ B.\ Fuchs},
       DOI = {10.2307/2373615},
       URL = {https://doi.org/10.2307/2373615},
}

@book {Chavel,
    AUTHOR = {Chavel, Isaac},
     TITLE = {Eigenvalues in {R}iemannian geometry},
    SERIES = {Pure and Applied Mathematics},
    VOLUME = {115},
      NOTE = {Including a chapter by Burton Randol,
              With an appendix by Jozef Dodziuk},
 PUBLISHER = {Academic Press, Inc., Orlando, FL},
      YEAR = {1984},
     PAGES = {xiv+362},
      ISBN = {0-12-170640-0},
   MRCLASS = {58G25 (35P99 53C20)},
  MRNUMBER = {768584},
MRREVIEWER = {G\'erard\ Besson},
}

@incollection {LiSe,
    AUTHOR = {Liskevich, V. A. and Semenov, Yu.\ A.},
     TITLE = {Some problems on {M}arkov semigroups},
 BOOKTITLE = {Schr\"odinger operators, {M}arkov semigroups, wavelet
              analysis, operator algebras},
    SERIES = {Math. Top.},
    VOLUME = {11},
     PAGES = {163--217},
 PUBLISHER = {Akademie Verlag, Berlin},
      YEAR = {1996},
      ISBN = {3-05-501710-2},
   MRCLASS = {47D07 (47A55 47B38)},
  MRNUMBER = {1409835},
MRREVIEWER = {Peter\ Stollmann},
}

\end{document}